\documentclass[10pt, reqno]{amsart}
\usepackage{amsfonts}

\allowdisplaybreaks
\usepackage{amssymb,amsthm,amsmath,bm}
\usepackage[numbers,sort&compress]{natbib}
\usepackage{color}

\title[2D magneto-micropolar equations]{Initial-boundary value problem for 2D magneto-micropolar equations with zero angular viscosity}
\author{Shasha Wang, Wen-Qing Xu, Jitao Liu$^*$}
\address[Shasha Wang]{College of Applied Sciences, Beijing University of Technology, Beijing, 100124, P. R. China.}
\email{wshasha@emails.bjut.edu.cn}
\address[Wen-Qing Xu] {
Department of Mathematics and Statistics,
  California State University, Long Beach, CA 90840, USA.}
\email{wxu@csulb.edu}
\address[Jitao Liu]{College of Applied Sciences, Beijing University of Technology, Beijing, 100124, P. R. China.}
\email{jtliu@bjut.edu.cn,\,jtliumath@qq.com}

\thanks{$^*$Corresponding author}

\theoremstyle{plain}
\newtheorem{corollary}{Corollary}[section]
\newtheorem{theorem}{Theorem}[section]
\newtheorem{lemma}{Lemma}[section]
\newtheorem{proposition}{Proposition}[section]

\theoremstyle{definition}
\newtheorem{definition}{Definition}[section]

\let\f=\frac
\let\p=\partial
\def\R{\mathbb R}

\newcommand{\beq}{\begin{equation}}
\newcommand{\eeq}{\end{equation}}
\newcommand{\ben}{\begin{eqnarray}}
\newcommand{\een}{\end{eqnarray}}
\newcommand{\beno}{\begin{eqnarray*}}
\newcommand{\eeno}{\end{eqnarray*}}

\begin{document}

\begin{abstract}
In this paper, we are concerned with the initial-boundary value problem to the 2D magneto-micropolar system with zero angular viscosity in a smooth bounded domain. We prove that there exists a unique global strong solution of such a system by imposing natural boundary conditions and regularity assumptions on the initial data, without any compatibility condition.
\end{abstract}
\maketitle

\section{Introduction}
\label{intro}
\setcounter{equation}{0}

This paper is concerned with the initial-boundary value problem for the incompressible magneto-micropolar equations, which were first studied in 1974 by Ahmadi and Shahinpoor in \cite{AS}.
The standard 3D magneto-micropolar equations can be written as
\begin{equation}\label{3D}
\left\{\begin{array}{ll}
\p_t {\mathbf u} + ({\mathbf u}\cdot\nabla) {\mathbf u} + \nabla(\pi + \frac{1}{2}|{\mathbf b}|^2) = (\mu + \chi) \Delta {\mathbf u} + ({\mathbf b} \cdot \nabla) {\mathbf b} + \chi (\nabla \times {\mathbf w}),\\
\p_t{\mathbf w} + ({\mathbf u}\cdot\nabla) {\mathbf w} + 2\chi {\mathbf w} =
 \gamma \Delta {\mathbf w} + (\alpha + \beta)\nabla {\rm div}{\mathbf w} + \chi (\nabla\times {\mathbf u}),\\
\p_t{\mathbf b} + ({\mathbf u}\cdot\nabla) {\mathbf b} = \nu \Delta {\mathbf b} + ({\mathbf b}\cdot\nabla) {\mathbf u},\\
\nabla \cdot {\mathbf u}=0, \quad \nabla \cdot {\mathbf b} = 0,
\end{array}\right.
\end{equation}
where
$ {\mathbf u} = {\mathbf u}(x,t) $, $ {\mathbf w}={\mathbf w}(x,t) $, $ {\mathbf b} = {\mathbf b}(x,t) $ and $ \pi = \pi(x,t) $ denote the fluid velocity, the micro-rotation velocity (angular velocity of the rotation of the fluid particles),
the magnetic field and the pressure respectively. The constant $ \mu $ denotes the kinematic viscosity, $ \chi $ the vortex viscosity, $ {\nu}$ the magnetic diffusivity, and $ \alpha $, $ \beta $, $ \gamma $ the angular viscosities, all of which are assumed to be positive.

The magneto-micropolar system is closely related to many classical systems. When the micro-rotation and magnetic effects are neglected, namely $ {\mathbf w} = {\mathbf b} = 0 $, the system (\ref{3D}) reduces to the incompressible Navier-Stokes equations.
The incompressible Navier-Stokes equations, which govern the motion of the incompressible fluid in a domain, are widely used in engineering and studied in mathematics.
When $ \chi = 0 $ and $ \mathbf w = 0 $, the system (\ref{3D}) reduces to the incompressible magnetohydrodynamics (MHD) system that describes the motion of electrically conducting fluids. It is widely used in astrophysics, geophysics, plasma physics and other applied sciences. Owing to its physical applications and mathematical significance, a lot of mathematicians have been dedicated to the mathematical study of the MHD system and important progress has been made in the past decades
(see, e.g., \cite{CWY, CMZ, DuLions, HX, HX2, JLiu, KL, Lei2, LZ, PZZ, RWXZ, SeTe, WWX, Zhang}).
Finally, when $ {\mathbf b} = 0 $, the system (\ref{3D}) reduces to the micropolar fluids, which are proposed by Eringen \cite{Er}.
The micropolar fluids are fluids with microstructure, e.g. anisotropic fluids, such as liquid crystals made up of dumbbell molecules and animal blood. In particular, the global well-posedness of the micropolar fluids has attracted extensive attention (see, e.g., \cite{DC, DLW2, DWX, DZ, GaRi, JLWY, Li, LW2018, SP, YN}).

The magneto-micropolar system can model the motion of incompressible and electrically conductive micropolar fluids with rigid microinclusions in a magnetic field.
The interaction between the flow field and the magnetic field is manifested through the body force and body couple.
Specifically, the magneto-micropolar system has constant density and electrical conductivity, while the relativistic, Hall and temperature effects are ignored.
In addition, such a system finds application in magnetohydrodynamics (MHD) generators with neutral fluid seedings and neutral fissionable bacteria in the form of rigid microinclusions (see \cite{AS}). Because of its physical applications, the study on this system has attracted much more attention.

Starting from \cite{AS}, where the Serrin-type criteria for the magneto-micropolar equations were established, this system has been studied extensively (see, e.g., \cite{OTRM, RoMe, TWZ, Yuan}). In \cite{RoMe}, the local in time existence and uniqueness of strong solutions were obtained by the spectral Galerkin method in 2D and 3D spaces. The global existence of strong solutions with small initial data was obtained in 3D space \cite{OTRM}. The authors in \cite{RmBo} proved the existence of weak solutions by the Galerkin method in 2D and 3D spaces, and in the 2D case, proved the uniqueness of weak solutions. Recently, more attention is focused on the 2D magneto-micropolar equations. As a matter of fact, for $ {\bf x} = (x_1, x_2) \in \mathbb{R}^2 $, by setting
\begin{equation*}
  {\mathbf u}=(u_1, u_2, 0),\quad
  {\mathbf w} = (0, 0, w),\quad
  {\mathbf b}=(b_1, b_2, 0),\quad
   \pi = \pi({\bf x}, t),
\end{equation*}
the 3D magneto-micropolar equations reduce to the 2D magneto-micropolar equations,
\begin{equation}\label{2D}
\left\{\begin{array}{ll}
{\mathbf u}_t + ({\mathbf u} \cdot \nabla) {\mathbf u} + \nabla p
= (\mu + \chi) \Delta {\mathbf u}
  +({\mathbf b} \cdot \nabla) {\mathbf b} - \chi \nabla^{\perp} w,\\
{w}_t + ({\mathbf u}\cdot\nabla) w + 2\chi w = \gamma \Delta w + \chi \nabla^{\perp}\cdot {\mathbf u},\\
{\mathbf b}_t + ({\mathbf u} \cdot \nabla) {\mathbf b}= \nu \Delta{\mathbf b} + ({\mathbf b} \cdot \nabla) {\mathbf u},\\
\nabla \cdot {\mathbf u} = 0,\quad \nabla \cdot {\mathbf b} = 0.
\end{array}\right.
\end{equation}
Here $ p= \pi + \frac{1}{2}|{\mathbf b}|^2 $, $ {\mathbf u} = (u_1({\bf x},t), u_2({\bf x},t)) $ and $ {\mathbf b} = (b_1({\bf x},t), b_2({\bf x},t)) $ are 2D vector fields with the corresponding scalar vorticities given by $ \Phi \equiv {\nabla}^{\perp}\cdot {\mathbf u} = \p_1{u_2} - \p_2{u_1} $  and $ \Psi \equiv {\nabla}^{\perp}\cdot {\mathbf b} = \p_1{b_2} - \p_2{b_1} $, while $ w $ represents a scalar function with
$ {\nabla^{\perp}}w = (-\p_2 w, \p_1 w) $.

We remark that the global regularity to the inviscid magneto-micropolar equations is still {\it a challenging open} problem.
Therefore, it is natural to study the intermediate cases such as partial viscosities. It's worth noting that more recent attention is focused on the 2D partial viscosity cases (e.g., \cite{ GuoSh, YK2015}). Recently, for the case in $ \mathbb{R}^2 $
where $ \mu, \chi, \nu > 0 $ and $ \gamma = 0 $, i.e.,
\begin{equation}\label{eq1}
\left\{\begin{array}{ll}
{\mathbf u}_t + ({\mathbf u} \cdot \nabla) {\mathbf u} + \nabla p
= (\mu + \chi) \Delta {\mathbf u}
  +({\mathbf b} \cdot \nabla) {\mathbf b} - \chi \nabla^{\perp} w,\\
{w}_t + ({\mathbf u}\cdot\nabla) w + 2\chi w= \chi \nabla^{\perp}\cdot {\mathbf u},\\
{\mathbf b}_t + ({\mathbf u} \cdot \nabla) {\mathbf b}= \nu \Delta{\mathbf b} + ({\mathbf b} \cdot \nabla) {\mathbf u},\\
\nabla \cdot {\mathbf u} = 0,\quad \nabla \cdot {\mathbf b} = 0,
\end{array}\right.
\end{equation}
the author \cite{YK2015} proved the global regularity for the Cauchy problem of the system \eqref{eq1}. The main obstacle comes from the micro-rotational term $ - \chi \nabla^{\perp}w $ in the equation $ (\ref{eq1})_1 $, which brings great difficulties in obtaining any high order estimates. To overcome this difficulty, the author in \cite{YK2015} considered the combined quantity
\begin{equation*}
  Z = \Phi - \f{\chi}{\mu + \chi} w,
\end{equation*}
which satisfies the transport-diffusion equation
\begin{equation*}\label{Z}
\p_t Z - (\mu + \chi) \Delta Z + {\mathbf u} \cdot \nabla Z + \f{\chi^2}{\mu + \chi} Z - \left(\f{2\chi^2}{\mu + \chi}- \f{\chi^3}{(\mu + \chi)^2}\right)w - \mathbf b \cdot \nabla \Psi = 0.
\end{equation*}
Thanks to this equation, one can first derive a bound on
$ \| Z \|_{L^1_t L^\infty_x} $ through establishing a suitable bound on
$ \Psi = \nabla^\bot \cdot \mathbf b $.
Based on this fact, the author further obtains the estimates of
$ \| w \|_{L^\infty_t L^\infty_x} $ and
$ \| \Phi \|_{L^1_t L^\infty_x} $,
which help to establish the global regularity of strong solutions.

However, for the system \eqref{eq1}, the corresponding initial-boundary value problem is still open. In fact, in many real world applications, flows are often restricted to bounded domains with appropriate conditions on the boundary, and these applications naturally lead to the studies of the initial-boundary value problems. In addition, solutions of the initial-boundary value problems may exhibit much richer phenomena than that of the whole space.

In this paper, we will consider the initial-boundary value problem of the system \eqref{eq1} with Dirichlet boundary conditions
\begin{equation}\label{eq2}
  {\mathbf u}|_{\p \Omega} = 0, \quad
  {\mathbf b}|_{\p \Omega} = 0,
\end{equation}
and initial conditions
\begin{equation}\label{eq20}
({\mathbf u}, w, {\mathbf b})(x,0)
 = ({\mathbf u}_0, w_0, {\mathbf b}_0)(x),
 \quad\,\hbox{in}\,\,\Omega, \\
\end{equation}
where $\Omega\subset \R^2$ represents a bounded domain with smooth boundary. Our goal is, without any compatibility condition, to establish the global existence and uniqueness of strong solutions to the system \eqref{eq1}--\eqref{eq20}.
Our main results are stated as follows.

\begin{theorem}\label{T1}
Let $ \Omega \subset \R^2 $ be a bounded domain with smooth boundary.
Suppose that the initial data $ ({\mathbf u}_0, w_0, {\mathbf b}_0) $ satisfies
\begin{equation*}
{\mathbf u}_0 \in H_0^1(\Omega) \cap H^{2}(\Omega), \quad
 w_0\in W^{1,4}(\Omega), \quad
{\mathbf b}_0 \in H_0^1(\Omega),
\end{equation*}
then there exists a unique strong solution
$ ({\mathbf u}, w, {\mathbf b}) $ of the system \eqref{eq1}--\eqref{eq20} globally in time,
such that
\begin{eqnarray*}
& & {\mathbf u} \in L^\infty(0, T; H_0^1(\Omega))
     \cap L^2(0, T; W^{2,4}(\Omega)),\quad \sqrt{t}\, {\mathbf u} \in L^\infty(0, T; H^2(\Omega)),\\
& & w \in L^\infty(0, T; W^{1,4}(\Omega)),\\
& & {\mathbf b} \in L^\infty(0, T; H_0^1(\Omega))
     \cap L^2(0, T; H^2(\Omega)),\quad \sqrt{t}\, {\mathbf b} \in L^\infty(0, T; H^2(\Omega)),
  \label{regclass}
\end{eqnarray*}
hold for any $ T > 0 $.
\end{theorem}

{\em Remarks:
(i). Theorem 1.1 extends the corresponding results in \cite{LW2018} from the micropolar system to the magneto-micropolar system. In fact, compared with the micropolar system,
the presence of the magnetic fields will bring up much stronger coupling in the nonlinearities. To overcome the difficulties, we work out more delicate a priori estimates.

(ii). More importantly, the compatibility condition on the initial data plays an important role in \cite{LW2018}, while in this paper, no compatibility condition is required.}

\vskip .1in

The proof of Theorem \ref{T1} consists of three main parts. In the first part, we will establish the global existence of weak solutions to the system \eqref{eq1}--\eqref{eq20} in the following sense.

\begin{definition}\label{weak}
Let $\Omega\subset\R^2$ be a bounded domain with smooth boundary.
A triple $ ({\mathbf u}, w, {\mathbf b}) $ of measurable functions is called a weak solution of the system (\ref{eq1})--(\ref{eq20}) if
\begin{eqnarray*}
&&(1)~~ {\mathbf u} \in C(0,T;L^2(\Omega)) \cap L^2(0,T;H_0^1(\Omega)), \,\,\,
        {\mathbf b} \in C(0,T;L^2(\Omega)) \cap L^2(0,T;H_0^1(\Omega)), \,\,\, w \in C(0,T;L^4(\Omega));\\
&&(2) \int_{\Omega}{\mathbf u}_0\cdot{\boldsymbol \varphi}_0 \,dx
    + \int_0^T\int_{\Omega}\big[{\mathbf u}\cdot{\boldsymbol \varphi}_t
    - (\mu+\chi)\nabla {\mathbf u}:\nabla{\boldsymbol \varphi}
    + {\mathbf u}\cdot\nabla{\boldsymbol \varphi}\cdot {\mathbf u}
    - {\mathbf b}\cdot\nabla{\boldsymbol \varphi}\cdot {\mathbf b} + \chi w\nabla^{\perp}\cdot{\boldsymbol \varphi}\big]\,dxdt
      = 0,\\
&&~~~~~~~\quad \int_{\Omega}w_0\psi_0 \, dx + \int_0^T\int_{\Omega}\big[w\psi_t
    - 2\chi w\psi + {\mathbf u}\cdot\nabla\psi w
    - \chi {\mathbf u}\cdot\nabla^{\perp}\psi\big] \, dxdt = 0,\\
&&~~~~~~~\quad \int_{\Omega}{\mathbf b}_0\cdot{\boldsymbol \phi}_0 \, dx
    + \int_0^T\int_{\Omega}\big[{\mathbf b}\cdot{\boldsymbol \phi}_t
    - \nu \nabla {\mathbf b}:\nabla{\boldsymbol  \phi}
    + {\mathbf u}\cdot\nabla{\boldsymbol \phi}\cdot {\mathbf b}
    - {\mathbf b}\cdot\nabla{\boldsymbol \phi}\cdot {\mathbf u}\big]\,dxdt = 0;\\
&&(3)~~
     \int_{\Omega} \mathbf u \cdot \nabla \Theta \, dx = 0, \,\,\,
     \int_{\Omega} \mathbf b \cdot \nabla \Theta \, dx = 0,
     \quad {\rm for \ each} \ t \in [0, T);
\end{eqnarray*}
hold for any test vector fields $ {\boldsymbol{\varphi}} \in C_0^\infty([0,T)\times \Omega)^2 $ with $ \nabla\cdot{\boldsymbol \varphi} = 0 $
and $ {\boldsymbol{\phi}} \in C_0^\infty([0,T)\times \Omega)^2 $ with $ \nabla\cdot{\boldsymbol \phi} = 0 $,
any test functions $ \psi\in C^\infty([0,T)\times \Omega) $ and
$ \Theta \in C_0^\infty(\Omega) $, where $ A:B $ denotes the scalar matrix product
$ A:B = \sum\limits_{i,j} a_{ij} b_{ij} $.
\end{definition}

In the second part, we will build up high order estimates of the solutions under the initial and boundary conditions \eqref{eq2}--\eqref{eq20}. This will help us to show that the weak solutions obtained in the first part are actually strong solutions. In the last part, we prove the uniqueness of strong solutions.

\vskip 3mm

The main difficulty to obtain high order estimates still arises from the micro-rotational term $ - \chi \nabla^{\perp}w $. We note that the initial-boundary value problem on \eqref{eq1} is quite different from the Cauchy problem in \cite{YK2015} and is more complicated.
For the initial-boundary value problem, the transport-diffusion equation satisfied by $ Z $ would not work any more because of the presence of no-slip boundary condition for $ \mathbf u $.
However, a key observation in \cite{LW2018} was to introduce an auxiliary field $ \mathbf v $ which is at the energy level of one order lower than $w$ and choose an appropriate boundary condition for $ \mathbf v $.
It then provides us the cornerstone of establishing high order estimates. In this work, we will adopt similar ideas to establish high order estimates. To begin with, we introduce the vector field
\begin{equation*}
{\mathbf v} = - \f{\chi}{\mu + \chi} A ^{-1}{\nabla}^{\perp}w
\end{equation*}
to be the unique solution of the following stationary Stokes system with source term $ - \f{\chi}{\mu + \chi} {\nabla}^{\perp}w $ and the Stokes operator $ A $,
\begin{equation}\label{stokesv}
\left\{\begin{array}{ll}
- \Delta {\mathbf v} + \nabla p
= - \f{\chi}{\mu + \chi} {\nabla}^{\perp}w \quad&\hbox{in}\,\,\Omega,\\
\nabla\cdot {\mathbf v}=0\quad&\hbox{in}\,\,\Omega,\\
{\mathbf v}=0\quad&\hbox{on}\,\,\p{\Omega}.
\end{array}\right.
\end{equation}
Due to \eqref{stokesv}, after taking the operator $ A^{-1}\nabla^{\perp} $ on
$ \eqref{eq1}_2 $, it is clear that the field $ \mathbf v $ also solves
\ben\label{vt}
\p_t{\mathbf v} + 2 \chi {\mathbf v}
 - \f{\chi}{\mu + \chi}A^{-1}\nabla^{\perp}({\mathbf u}\cdot\nabla{w})
 + \frac{\chi^2}{\mu + \chi} A^{-1}\nabla^{\perp}(\nabla^{\perp}\cdot {\mathbf u})=0.
\een
Based on \eqref{eq1}, \eqref{stokesv} and \eqref{vt}, we further introduce a new field $ {\mathbf g} = {\mathbf u} - {\mathbf v} $ that satisfies the system
\begin{equation}\label{mathg}
\left\{\begin{array}{ll}
\p_t{\mathbf g} - (\mu + \chi) \Delta {\mathbf g} + (1- \mu - \chi)\nabla p
   = {\mathbf Q}\quad&\hbox{in}\,\,\Omega,\\
\nabla\cdot {\mathbf g}=0\quad&\hbox{in}\,\,\Omega,\\
{\mathbf g}=0\quad&\hbox{on}\,\,\p{\Omega},
\end{array}\right.
\end{equation}
where $ {\mathbf Q} = -{\mathbf u} \cdot \nabla{\mathbf u} + {\mathbf b} \cdot \nabla{\mathbf b} - \f{\chi}{\mu + \chi} A^{-1}\nabla^{\perp}({\mathbf u}\cdot\nabla{w}) + \f{\chi^2}{\mu + \chi} A^{-1} \nabla^{\perp}(\nabla^{\perp} \cdot {\mathbf u}) + 2\chi {\mathbf v} $.
Naturally, the establishment of this system overcomes the difficulty caused by the micro-rotational term $ - \chi \nabla^{\perp}w $ and the appearance of no-slip boundary condition for $ \mathbf u $.
\vskip 2mm

To guarantee the existence and uniqueness of strong solutions, we shall focus on deriving high order estimates. Motivated by \cite{Li2017}, to drop the compatibility condition on the initial data, we carry out the $t$-weighted $ H^2 $ estimates to the systems
\eqref{eq1} and \eqref{mathg} instead of $ H^2 $ estimates, namely
\begin{eqnarray*}
\sup\limits_{0 \leq t \leq T} \, t(\|\nabla^2{\mathbf g}, \nabla^2{\mathbf b} \|_{L^2(\Omega)}^2 + \|\p_t{\mathbf g}, \p_t{\mathbf b}\|_{L^2(\Omega)}^2 )
+ \int_0^T t \|\nabla\p_t{\mathbf g}, \nabla\p_t{\mathbf b} \|_{L^2(\Omega)}^2 \, dt
\leq C,
\end{eqnarray*}
see Proposition \ref{ubtH2} for details. We remark that above constant $ C $ is independent of $ \| \mathbf b_0 \|_{H^2(\Omega)} $, due to the weighted factor $ t $.
That is to say, as we establish the $ H^2 $ estimates of the magnetic field $\mathbf b$, the requirement on its initial data $\mathbf b_0$ is only $ H^1 $ norm. Subsequently, considering that $ \mathbf v $ is at the energy level of one order lower than $ w $, after establishing a priori estimates for $ \mathbf g $, we can successfully establish the desired high order estimates of $ \mathbf u $, see Proposition \ref{uH1} and Proposition \ref{ubtH2} below for details.

\vskip .1in

The remainder of the paper is organized as follows. First we state some useful preliminary tools and results for bounded domain in Section \ref{prel}. Then we prove the global existence of weak solutions in Section \ref{weak solutions}. Finally, we establish high order estimates and prove the uniqueness of strong solutions in Section \ref{strong solutions}.

\vskip .3in

\section{Preliminaries}
\label{prel}
\setcounter{equation}{0}

\vskip 1mm

In this section, we state some preliminary tools which will play a fundamental role in later sections.
We start with the Gagliardo-Nirenberg interpolation inequality for bounded domains (see, e.g., \cite {NIR}).

\begin{lemma}\label{P1}
Let $ \Omega \subset \R^n $ be a bounded domain with smooth boundary. Let $ 1 \leq p, q, r \leq \infty $, $ \alpha > 0 $ and $ j < m $ be nonnegative integers such that

$$\frac{1}{p} - \frac{j}{n} = \alpha\,\left( \frac{1}{r} - \frac{m}{n} \right) + (1 - \alpha)\frac{1}{q}, \qquad \frac{j}{m} \leq \alpha \leq 1,
$$
then every function $ f: \Omega \mapsto \mathbb{R} $ that lies in $ L^q(\Omega) $ with $m^{\rm th}$ derivative in
$ L^r(\Omega) $ also has $j^{\rm th}$ derivative in
$ L^p (\Omega) $. Furthermore, it holds that
$$ \| \mathrm{D}^{j} f \|_{L^{p}(\Omega)} \leq C_{1} \| \mathrm{D}^{m} f \|_{L^{r}(\Omega)}^{\alpha} \| f \|_{L^{q}(\Omega)}^{1 - \alpha} + C_{2} \|f\|_{L^{s}(\Omega)}, $$
where $ s>0 $ is arbitrary and the constants $C_1$ and $C_2$ depend upon $\Omega$ and the indices $ n, m, j, q, r, s, \alpha $ only.
\end{lemma}

\begin{corollary}\label{C1}
Let $ f: \mathbb{R}^2 \mapsto \mathbb{R} $ and suppose $\Omega \subset\R^2$ be a bounded domain with smooth boundary, then
\vskip 0.2cm
\begin{enumerate}
  \item $ \| f \|_{L^{4}(\Omega)} \leq C\, (\|  f \|_{L^{2}(\Omega)}^{\f12} \| \nabla f \|_{L^{2}(\Omega)}^{\f12} + \| f \|_{L^{2}(\Omega)}),\quad \forall f\in H^1(\Omega) $;
  \item $ \| \nabla f \|_{L^{4}(\Omega)} \leq C\, (\|  f \|_{L^{2}(\Omega)}^{\f14} \| \nabla^2 f \|_{L^{2}(\Omega)}^{\f34} + \| f \|_{L^{2}(\Omega)}),\quad   \forall f\in H^2(\Omega) $;
  \item $ \| f \|_{L^{\infty}(\Omega)} \leq C\, (\|  f \|_{L^{2}(\Omega)}^{\f12} \| \nabla^2 f \|_{L^{2}(\Omega)}^{\f12} + \| f \|_{L^{2}(\Omega)}),\quad \forall f\in H^2(\Omega) $;
  \item $ \| f \|_{L^{\infty}(\Omega)} \leq C\, (\|  f \|_{L^{2}(\Omega)}^{\f23} \| \nabla^3 f \|_{L^{2}(\Omega)}^{\f13} + \| f \|_{L^{2}(\Omega)}),\quad \forall f\in H^3(\Omega) $.
\end{enumerate}

\end{corollary}

\vskip .1in
The next two lemmas provide some regularity
estimates for the elliptic equations and Stokes system
defined on bounded domains (see, e.g., \cite{evans,GT,Galdi,Lady0,SunZhang}).

\begin{lemma}\label{elliptic}
Let $\Omega\subset\R^2$ be a bounded domain with smooth boundary. Consider the elliptic boundary value problem
\begin{equation}\label{eel}
\left\{\begin{array}{ll}
-\Delta f=g\quad&\hbox{in}\,\,\Omega,\\
f=0\quad&\hbox{on}\,\,\p{\Omega}.
\end{array}\right.
\end{equation}
If $g\in W^{m,p}(\Omega)$ for some $p\in(1,\infty)$, then the system (\ref{eel}) has a unique solution $f$ satisfying
$$
\|f\|_{W^{m+2,p}(\Omega)}\leq C\|g\|_{W^{m,p}(\Omega)},
$$
where $m\geq -1$ be an integer and the constant $C$ depends only on $\Omega, \,m$ and $p$.
\end{lemma}

\begin{lemma}\label{stokes}
Let $\Omega\subset\R^n$ be a bounded domain with smooth boundary. Consider the stationary Stokes system
\begin{equation}\label{stokesequl}
\left\{\begin{array}{ll}
-\Delta {\mathbf u}+\nabla p={\mathbf f}\quad&\hbox{in}\,\,\Omega,\\
\nabla\cdot {\mathbf u}=0\quad&\hbox{in}\,\,\Omega,\\
{\mathbf u}=0\quad&\hbox{on}\,\,\p{\Omega}.
\end{array}\right.
\end{equation}
If ${\mathbf f}\in L^{q}(\Omega)$, $q\in(1,\infty)$,  then there exists a unique solution ${\mathbf u}\in  W_0^{1,q}(\Omega)\cap W^{2,q}(\Omega)$ of the system (\ref{stokesequl}) satisfying
\ben\label{Stokes1}
\|\mathbf u\|_{W^{2,q}(\Omega)}+\|\nabla p\|_{L^{q}(\Omega)}\leq C\|\mathbf f\|_{L^{q}(\Omega)}.
\een
If ${\mathbf f}=\nabla\cdot F$ with $F\in L^{q}(\Omega)$, $q\in(1,\infty)$, then
\ben\label{Stokes2}
\|\mathbf u\|_{W^{1,q}(\Omega)}\leq C\|F\|_{L^{q}(\Omega)}.
\een
If ${\mathbf f}=\nabla\cdot F$ with $F_{ij}=\p_{k}H^{k}_{ij}$ and $H^{k}_{ij}\in W_0^{1,q}(\Omega)$ for $i,j,k=1,...,n$, $q\in(1,\infty)$, then
\ben\label{Stokes3}
\|\mathbf u\|_{L^{q}(\Omega)}\leq C\|H\|_{L^{q}(\Omega)},
\een
where all the above constants $C$ depend only on $\Omega$ and $q$.

Besides, if ${\mathbf f}=\nabla\cdot F$ with $F\in W^{1,q}(\Omega)$, $q\in(2,\infty)$, then
\ben\label{newstokes1}
\|\nabla{\mathbf u}\|_{L^{\infty}(\Omega)}\leq C(1+\|F\|_{L^\infty(\Omega)}){\rm ln}(e+\|\nabla{F}\|_{L^{q}(\Omega)}),
\een
where the constant $C$ depends only on $\Omega$.
\end{lemma}

As stated in the introduction, $\mathbf v$ is at the energy level of one order lower than $w$.  For the system \eqref{stokesv}, we set
\begin{align}\notag
{\mathbf f}= {\mathbf v},\quad
F = \f{\chi}{\mu + \chi}\left( {\begin{array}{*{20}c}
	0&w   \\
	-w&0 \\
	\end{array} } \right).
\end{align}
Then by recalling the inequality \eqref{Stokes2},
it is clear that the estimate
\begin{equation}\label{keyestimates}
\|\mathbf v\|_{W^{1,q}(\Omega)}\leq C\|w\|_{L^{q}(\Omega)},
\end{equation}
holds for any $ q\in (1,\infty) $,  where the constant $C$ depends only on $\Omega$ and $q$.
Similarly, by the inequality \eqref{newstokes1}, we have the following estimate, for $ q \in (2,\infty) $,
\ben\label{keyestimates1}
\|\nabla{\mathbf v}\|_{L^{\infty}(\Omega)}\leq C(1+\|w\|_{L^\infty(\Omega)}) \ln (e + \|\nabla{w}\|_{L^{q}(\Omega)}),
\een
where the constant $C$ depends only on $\Omega$.

\begin{lemma}\label{timestokes}
Let $1<p,\,q<\infty$, and suppose that $f\in L^p(0,T;L^q(\Omega))$, $u_0\in W^{2,p}(\Omega)$. If $({\mathbf u},\,p)$ is a solution of the Stokes system
\begin{equation}\label{stokesequ2}
\left\{\begin{array}{ll}
\p_t{\mathbf u}-\Delta {\mathbf u}+\nabla p={\mathbf f}\quad&\hbox{in}\,\,\Omega,\\
\nabla\cdot {\mathbf u}=0\quad&\hbox{in}\,\,\Omega,\\
{\mathbf u}=0\quad&\hbox{on}\,\,\p{\Omega},\\
{\mathbf u}(x,0)={\mathbf u}_0(x)\quad&\hbox{in}\,\,\Omega,
\end{array}\right.
\end{equation}
then it holds that
\ben\label{timestokes1}
\|\p_t{\mathbf u},\,{\nabla}^2{\mathbf u},\,\nabla p\|_{L^p(0,T;L^{q}(\Omega))}\leq C(\|\mathbf f\|_{L^p(0,T;L^{q}(\Omega))}+\|{\mathbf u}_0\|_{W^{2,p}(\Omega)}),
\een
where the constant $C$ depends only on $p$, $q$ and $\Omega$.
\end{lemma}

\vskip 2mm

\section{Global weak solution}
\label{weak solutions}
\setcounter{equation}{0}

In this section, we focus on proving the global existence of weak solutions of the system (\ref{eq1})--(\ref{eq20}).
To this end, we start with the a priori estimates for the system (\ref{eq1})--(\ref{eq20}).
\begin{proposition}\label{uH1}
Let $ \Omega \subset \R^2 $ be a bounded domain with smooth boundary.
Assume that the initial data
$ ({\mathbf u}_0, w_0, {\mathbf b}_0) $ satisfies
$ {\mathbf u}_0 \in H_0^1(\Omega) \cap H^{2}(\Omega) $, $ w_0\in W^{1,4}(\Omega) $, $ {\mathbf b}_0 \in H_0^1(\Omega) $.
Let the triple $ ({\mathbf u}, w, {\mathbf b}) $ be a smooth solution of the system \eqref{eq1}--\eqref{eq20}.
Then it holds that
\begin{eqnarray*}
 & & \sup\limits_{0 \leq t\leq T}
 (\|{\mathbf u}\|^2_{H^1(\Omega)}  + \|w\|^2_{W^{1,4}(\Omega)}+ \|{\mathbf b}\|^2_{H^1(\Omega)})
+ \int_0^T(\|{\nabla^2\mathbf u}\|_{L^4(\Omega)}^2+\|\nabla^2 {\mathbf b}\|_{L^2(\Omega)}^2)\,dt\\
& & +\int_0^T(\|\p_t{\mathbf u}\|^2_{L^2(\Omega)}+ \|\p_t w\|^2_{L^4(\Omega)}+\|\p_t{\mathbf b}\|^2_{L^2(\Omega)})\,dt
\leq C,
\end{eqnarray*}
where the constant $ C $ depends only on $\Omega$, $ T $, $ \|{\mathbf u}_0\|_{H^{2}(\Omega)} $, $ \|{w}_0\|_{W^{1,4}(\Omega)} $, $ \|{\mathbf b}_0\|_{H^{1}(\Omega)} $.
\end{proposition}

The proof of Proposition \ref{uH1} is divided into the following subsections.
~\\
\subsection{$ L^2 $ estimates}
~\\
\label{section3.1}

We start with the basic energy estimates in the following Lemma.
\begin{lemma}\label{uL2}
Let $ \Omega \subset \R^2 $ be a bounded domain with smooth boundary and $ ({\mathbf u}, w, {\mathbf b}) $ be the smooth solution of system \eqref{eq1}--\eqref{eq20}.
If, in addition, $ {\mathbf u}_0 \in L^2(\Omega) $,
$ w_0 \in L^2(\Omega) $, $ {\mathbf b}_0 \in L^2(\Omega) $, then it holds that
\begin{eqnarray*}
  \sup\limits_{0 \leq t \leq T} (\|{\mathbf u}\|^2_{L^2} + \|w\|^2_{L^2} + \|{\mathbf b}\|^2_{L^2}) + (\mu + \chi)\int_0^t\|\nabla{\mathbf u}\|^2_{L^2}dt+ 4\chi\int_0^t\| w \|^2_{L^2}dt + 2 \nu \int_0^t\|\nabla{\mathbf b}\|^2_{L^2}dt \leq C,
\end{eqnarray*}
where the constant $C$ depends only on $\Omega$, $ T $, $ \|{\mathbf u}_0\|_{L^2(\Omega)} $, $ \|{w}_0\|_{L^{2}(\Omega)} $ and $ \|{\mathbf b}_0\|_{L^2(\Omega)} $.
\end{lemma}

\begin{proof}
We take the $ L^2 $-inner product of the system $\eqref{eq1}$ with $ {\mathbf u}, w, {\mathbf b} $ respectively to obtain
\begin{eqnarray}\label{L}
&& \f12 \f{d}{dt} (\|\mathbf u\|_{L^2(\Omega)}^2 + \|w\|_{L^2(\Omega)}^2 + \|\mathbf b\|_{L^2(\Omega)}^2) +(\mu + \chi)\|\nabla {\mathbf u}\|_{L^2(\Omega)}^2 + 2 \chi \|w\|_{L^2(\Omega)}^2+\nu \|\nabla {\mathbf b}\|_{L^2(\Omega)}^2   \nonumber\\
&=& - \chi \int_{\Omega}\nabla^{\perp}w\cdot {\mathbf u}\, dx
   + \chi \int_{\Omega}\nabla^{\perp}\cdot {\mathbf u}w\, dx,
\end{eqnarray}
where we have used the equalities
$$\int_{\Omega} \mathbf u \cdot \nabla\mathbf u \cdot \mathbf u \,dx =
\int_{\Omega} \mathbf u \cdot \nabla\mathbf b \cdot \mathbf b \,dx = 0,\,\,\,\int_{\Omega} \mathbf b \cdot \nabla\mathbf b \cdot \mathbf u \,dx +
 \int_{\Omega} \mathbf b \cdot \nabla\mathbf u \cdot \mathbf b \,dx = 0.$$

Note that
$ - \nabla^{\perp}w \cdot {\mathbf u} = u_1 \partial_{2} w - u_2 \partial_{1} w
= \partial_{2} (u_1\, w) - \partial_{1} (u_2\,w) + \, (\nabla^{\perp}\cdot {\mathbf u}) w $.
Then by using integration by parts and applying the boundary condition
$ {\mathbf u}|_{\p \Omega} = 0 $, for $\mathbf{n}^\perp=(-n_2,n_1)$, we have,
\ben
&& - \chi \int_{\Omega}\nabla^{\perp}w \cdot {\mathbf u}\,dx
   + \chi \int_{\Omega}\nabla^{\perp} \cdot {\mathbf u} w\,dx
= 2 \chi \int_{\Omega}\nabla^{\perp} \cdot {\mathbf u} w\,dx
   - \chi \int_{\p \Omega}{\mathbf u} \cdot{{\bm n}^\perp}w\,ds  \notag\\
&=& 2 \chi \int_{\Omega}\nabla^{\perp} \cdot {\mathbf u} w\,dx\leq \f{(\mu+\chi)}2 \|\nabla {\mathbf u}\|_{L^2(\Omega)}^2
    + C \|w\|_{L^2(\Omega)}^2,\label{bbb}
\een
which, together with (\ref{L}) and Gr\"{o}nwall's inequality, gives
\begin{eqnarray*}
\sup\limits_{0 \leq t \leq T} (\|{\mathbf u}\|^2_{L^2} + \|w\|^2_{L^2} + \|{\mathbf b}\|^2_{L^2}) + (\mu + \chi)\int_0^t\|\nabla{\mathbf u}\|^2_{L^2}dt
+ 4\chi\int_0^t\| w\|^2_{L^2}dt + 2 \nu \int_0^t\|\nabla{\mathbf b}\|^2_{L^2}dt
 \leq C,
\end{eqnarray*}
where the constant $ C $ depends only on $\Omega$, $ T $, $ \|{\mathbf u}_0\|_{L^2(\Omega)} $, $ \|{w}_0\|_{L^{2}(\Omega)} $ and $ \|{\mathbf b}_0\|_{L^2(\Omega)} $.
This completes the proof of Lemma \ref{uL2}.
\end{proof}
\vskip 0mm
~\\
\subsection{$ H^1 $ estimates}
~\\
\label{section3.2}

In this subsection, we establish the $ H^1 $ estimates of $ \mathbf u $ and $ \mathbf b $. By recalling $ {\mathbf g} = {\mathbf u} - {\mathbf v} $ and \eqref{keyestimates}, we have
\begin{equation*}
\|\mathbf v\|_{W^{1,q}(\Omega)}\leq C\|w\|_{L^{q}(\Omega)},
  \quad {\rm for~~any}~~ q \in (1,\infty),
\end{equation*}
which implies, after applying Lemma \ref{uL2}, that
\begin{eqnarray}\label{uh1}
  &&\|\mathbf u\|_{L^\infty(0,T;H^1(\Omega))}
\leq \|\mathbf g\|_{L^\infty(0,T;H^1(\Omega))}
     + \|\mathbf v\|_{L^\infty(0,T;H^1(\Omega))}   \nonumber\\
&\leq& C(\|\mathbf g\|_{L^\infty(0,T;H^1(\Omega))}
     + \| w\|_{L^\infty(0,T;L^2(\Omega))})
\leq C(\|\mathbf g\|_{L^\infty(0,T;H^1(\Omega))}
     + 1).
\end{eqnarray}
Therefore, to get the $ H^1$ estimates of $ {\mathbf u} $  and $ \mathbf b $, it suffices to establish the $ H^1$ estimates of $ {\mathbf g} $  and $ \mathbf b $ as below.

\begin{lemma}\label{gH1}
Under the assumptions of  Lemma \ref{uL2}, we further assume
$ \mathbf {u}_0 \in H^1(\Omega) $, $ \mathbf {b}_0 \in H^1(\Omega) $ and
$ w_0 \in L^4(\Omega) $, then it holds that
\begin{eqnarray*}
&&\sup\limits_{0 \leq t \leq T}(\|\nabla{\mathbf g}\|^2_{L^2(\Omega)} + \|\nabla{\mathbf b}\|^2_{L^2(\Omega)} + \|w\|^2_{L^4(\Omega)}) +
\int_0^t(\|\p_t{\mathbf g}\|^2_{L^2(\Omega)}
+ \|\p_t{\mathbf b}\|^2_{L^2(\Omega)}) \, dt  \\
&& + \int_0^t(\|\nabla^2{\mathbf g}\|^2_{L^2(\Omega)}+\|\nabla^2{\mathbf b}\|^2_{L^2(\Omega)} +  \|w\|^2_{L^4(\Omega)}) \, dt \leq C,
\end{eqnarray*}
where the constant $C$ depends only on $\Omega$, $ T $, $ \|{\mathbf u}_0\|_{H^1(\Omega)} $, $ \|{\mathbf b}_0\|_{H^1(\Omega)} $ and $ \|{w}_0\|_{L^{4}(\Omega)} $.
\end{lemma}

\begin{proof}
$\bf Step ~1 $.
First, taking the inner product of $ \eqref{mathg}_1 $ with $ \p_t \mathbf g $,
applying H\"{o}lder's inequality and Young's inequality, we obtain
\ben
&&\f{\mu + \chi}2 \f{d}{dt} \|\nabla\mathbf g\|_{L^2(\Omega)}^2
  + \|\p_t {\mathbf g}\|_{L^2(\Omega)}^2= \int_{\Omega}{\mathbf Q} \cdot \p_t{\mathbf g}\,dx
\leq \f12\|\p_t {\mathbf g}\|_{L^2(\Omega)}^{2} + \f12\|\mathbf Q\|_{L^2(\Omega)}^2.
\label{A100}
\een
By Lemma \ref{stokes}, we have
\begin{equation}\label{stokes g}
(\mu + \chi)\|\nabla^2{\mathbf g}\|_{L^2(\Omega)}^{2}
\leq C_1\|\p_t{\mathbf g}\|_{L^2(\Omega)}^{2} + C\|{\mathbf Q}\|_{L^2(\Omega)}^{2}.
\end{equation}
Multiplying (\ref{A100}) by $ 4 C_1 $, and summing the resultant with (\ref{stokes g}), one gets
\begin{equation}\label{A102}
2C_1(\mu + \chi)\f{d}{dt}\|\nabla{\mathbf g}\|_{L^2(\Omega)}^{2}
 + C_1\|\p_t{\mathbf g}\|_{L^2(\Omega)}^{2} +  (\mu + \chi)\|\nabla^2{\mathbf g}\|_{L^2(\Omega)}^{2}
 \leq C \|{\mathbf Q}\|_{L^2(\Omega)}^{2},
\end{equation}
with
\ben
\|\mathbf Q\|_{L^2(\Omega)}^2
&\leq& \|{\mathbf b}\cdot\nabla{\mathbf b}\|_{L^2(\Omega)}^2
      + \|{\mathbf u}\cdot\nabla{\mathbf u}\|_{L^2(\Omega)}^2
      + C\|A^{-1}\nabla^{\perp}({\mathbf u}\cdot\nabla{ w})\|_{L^2(\Omega)}^2\notag\\
&& +C\|A^{-1}\nabla^{\perp}(\nabla^{\perp}\cdot {\mathbf u})\|_{L^2(\Omega)}^2
   +C\|{\mathbf v}\|_{L^2(\Omega)}^2= \sum\limits_{i=1}^5 I_i.  \notag
\een

Subsequently, we will estimate the five terms one by one. By making use of H\"{o}lder inequality, Corollary \ref{C1}, Young's inequality and \eqref{keyestimates}, it follows that
\begin{eqnarray}
I_1
&\leq& \|{\mathbf b}\|^2_{L^4(\Omega)}
        \|\nabla{\mathbf b}\|^2_{L^4(\Omega)}\nonumber\\
&\leq& C \|{\mathbf b}\|_{L^2(\Omega)}
      \|\nabla{\mathbf b}\|_{L^2(\Omega)}(\|\nabla{\mathbf b}\|_{L^2(\Omega)}^2
       + \|\nabla{\mathbf b}\|_{L^2(\Omega)} \|\nabla^2{\mathbf b}\|_{L^2(\Omega)}) \nonumber\\
&\leq& \f{(\mu+\chi)}4\|\nabla^2{\mathbf b}\|_{L^2(\Omega)}^2 + C\|{\mathbf b}\|_{L^2(\Omega)}^2\|\nabla{\mathbf b}\|_{L^2(\Omega)}^4 + C\|{\mathbf b}\|_{L^2}\|\nabla{\mathbf b}\|_{L^2}\|\nabla{\mathbf b}\|_{L^2(\Omega)}^2, \label{A103}\\
I_2
&\leq& \|{\mathbf u}\cdot\nabla{\mathbf g}\|_{L^2(\Omega)}^2 +
       \|{\mathbf u}\cdot\nabla{\mathbf v}\|_{L^2(\Omega)}^2\notag\\
&\leq& \|{\mathbf u}\|_{L^4(\Omega)}^2\|\nabla{\mathbf g}\|_{L^4(\Omega)}^2 +
   \|{\mathbf u}\|_{L^4(\Omega)}^2\|\nabla{\mathbf v}\|_{L^4(\Omega)}^2\notag\\
&\leq& C\|{\mathbf u}\|_{L^2(\Omega)}\|\nabla{\mathbf u}\|_{L^2(\Omega)}\|\nabla{\mathbf g}\|_{L^2(\Omega)}^2+C\|{\mathbf u}\|_{L^2(\Omega)}\|\nabla{\mathbf u}\|_{L^2(\Omega)}\|\nabla{\mathbf g}\|_{L^2(\Omega)}\|\nabla^2{\mathbf g}\|_{L^2(\Omega)}\notag\\
&&+C\|{\mathbf u}\|_{L^2(\Omega)}\|\nabla{\mathbf u}\|_{L^2(\Omega)}\|w\|_{L^4(\Omega)}^2\notag\\
&\leq&\f{(\mu+\chi)}8\|\nabla^2 {\mathbf g}\|_{L^2(\Omega)}^{2}+C(\|{\mathbf u}\|_{L^2(\Omega)}\|\nabla{\mathbf u}\|_{L^2(\Omega)}+\|{\mathbf u}\|_{L^2(\Omega)}^2\|\nabla{\mathbf u}\|_{L^2(\Omega)}^2)\|\nabla{\mathbf g}\|_{L^2(\Omega)}^2\notag\\
&&+C\|{\mathbf u}\|_{L^2(\Omega)}\|\nabla{\mathbf u}\|_{L^2(\Omega)}\|w\|_{L^4(\Omega)}^2.
\label{A104}
\end{eqnarray}

It is clear that $ {\mathbf u}\cdot\nabla w = \nabla\cdot({\mathbf u}w) $ and $ {\mathbf u}w|_{\p\Omega} = 0 $, due to $ \nabla\cdot{\mathbf u} = 0 $ and $ {\mathbf u}|_{\p\Omega} = 0 $.
Therefore, by using Lemma \ref{stokes}, \eqref{keyestimates}, H\"{o}lder's inequality and Corollary \ref{C1},
we obtain
\ben
I_3 + I_4 + I_5
&\leq& C\|A^{-1}\nabla^{\perp}\nabla\cdot({\mathbf u}{w})\|_{L^2(\Omega)}^2+C\|A^{-1}\nabla^{\perp}(\nabla^{\perp}\cdot {\mathbf u})\|_{L^2(\Omega)}^2+C\|{\mathbf v}\|_{L^2(\Omega)}^2\notag\\
&\leq& C\|{\mathbf u}{w}\|_{L^2(\Omega)}^2+C\|{\mathbf u}\|_{L^2(\Omega)}^2+C\|w\|_{L^2(\Omega)}^2\notag\\
&\leq& C\|{\mathbf u}\|_{L^4(\Omega)}^2\|w\|_{L^4(\Omega)}^2+C\|{\mathbf u}\|_{L^2(\Omega)}^2+C\|w\|_{L^2(\Omega)}^2\notag\\
&\leq& C\|{\mathbf u}\|_{L^2(\Omega)}\|\nabla{\mathbf u}\|_{L^2(\Omega)}\|w\|_{L^4(\Omega)}^2+C(\|{\mathbf u}\|_{L^2(\Omega)}^2+\|w\|_{L^2(\Omega)}^2).
\label{A105}
\een
Summing up the estimates from \eqref{A103} to \eqref{A105}, one has
\begin{eqnarray}\label{A105.1}
\|\mathbf Q\|_{L^2(\Omega)}^2
&\leq& \f{(\mu+\chi)}8\|\nabla^2 {\mathbf g}\|_{L^2(\Omega)}^{2}+\f{(\mu + \chi)}4\|\nabla^2{\mathbf b}\|_{L^2}^2
+ C(1 + \|{\mathbf u}\|_{L^2(\Omega)}^2\|\nabla{\mathbf u}\|_{L^2(\Omega)}^2+ \|{\mathbf b}\|_{L^2(\Omega)}^2\|\nabla{\mathbf b}\|_{L^2(\Omega)}^2)\times  \nonumber \\
&& (\|\nabla{\mathbf g}\|_{L^2(\Omega)}^2 + \|\nabla{\mathbf b}\|_{L^2(\Omega)}^2 + \|w\|_{L^4}^2)  + C(\|{\mathbf u}\|_{L^2(\Omega)}^2 + \|w\|_{L^2}^2).
\end{eqnarray}
Then substituting \eqref{A105.1} into \eqref{A102} yields
\begin{eqnarray}\label{A106}
& &2C_1(\mu + \chi)\f{d}{dt}\|\nabla{\mathbf g}\|_{L^2(\Omega)}^{2}
 + C_1\|\p_t{\mathbf g}\|_{L^2(\Omega)}^{2} +  \f{7(\mu + \chi)}8\|\nabla^2{\mathbf g}\|_{L^2(\Omega)}^{2}  \notag\\
&\leq& \f{(\mu + \chi)}4\|\nabla^2{\mathbf b}\|_{L^2}^2
+ C(1 + \|{\mathbf u}\|_{L^2(\Omega)}^2\|\nabla{\mathbf u}\|_{L^2(\Omega)}^2
+ \|{\mathbf b}\|_{L^2(\Omega)}^2\|\nabla{\mathbf b}\|_{L^2(\Omega)}^2)\times \notag\\ &&(\|\nabla{\mathbf g}\|_{L^2(\Omega)}^2 + \|\nabla{\mathbf b}\|_{L^2(\Omega)}^2 + \|w\|_{L^4}^2) + C(\|{\mathbf u}\|_{L^2(\Omega)}^2 + \|w\|_{L^2}^2).
\end{eqnarray}

$ \bf Step ~2 $.
Repeating the above procedures for $ {\mathbf b} $,
we multiply $ \eqref{eq1}_3 $ by $ \p_t \mathbf b $ and by Lemma \ref{stokes} to obtain
\begin{eqnarray}
\label{A100b}
\f{\nu}2 \f{d}{dt} \|\nabla\mathbf b\|_{L^2(\Omega)}^2
     + \|\p_t {\mathbf b}\|_{L^2(\Omega)}^2
\leq \f12 \|\p_t {\mathbf b}\|_{L^2}^2
       + C \|\mathbf u \cdot \nabla\mathbf b\|_{L^2}^2
        + C \|\mathbf b \cdot \nabla\mathbf u\|_{L^2}^2,
\label{stokes b}
\end{eqnarray}
and
\begin{eqnarray}\label{stokes b1}
\nu \|\nabla^2{\mathbf b}\|_{L^2(\Omega)}^{2}
\leq C_1 \|\p_t{\mathbf b}\|_{L^2(\Omega)}^{2}
      + C \|\mathbf u \cdot \nabla\mathbf b\|_{L^2(\Omega)}^2
      + C \|\mathbf b \cdot \nabla\mathbf u\|_{L^2(\Omega)}^2.
\end{eqnarray}
Then by multiplying (\ref{A100b}) with $ 4 C_1 $ and summing the resultant with (\ref{stokes b1}), using H\"{o}lder's inequality, Corollary \ref{C1}, \eqref{keyestimates} and Young's inequality, it follows that
\begin{eqnarray} \label{A106.1}
 && 2C_1\nu\f{d}{dt}\|\nabla{\mathbf b}\|_{L^2(\Omega)}^2 + C_1\|\p_t{\mathbf b}\|_{L^2(\Omega)}^2 + \nu \|\nabla^2{\mathbf b}\|_{L^2(\Omega)}^2 \nonumber\\
 &\leq& C\|{\mathbf u}\cdot\nabla{\mathbf b}\|_{L^2(\Omega)}^2
    + C\|{\mathbf b}\cdot\nabla{\mathbf u}\|_{L^2(\Omega)}^2\nonumber\\
 &\leq& C \|{\mathbf u}\|_{L^4(\Omega)}^2 \|\nabla{\mathbf b}\|_{L^4(\Omega)}^2
    + C \|{\mathbf b}\|_{L^4(\Omega)}^2 \|\nabla{\mathbf u}\|_{L^4(\Omega)}^2\nonumber\\
 &\leq& C \|{\mathbf u}\|_{L^2(\Omega)}
      \|\nabla{\mathbf u}\|_{L^2(\Omega)}(\|\nabla{\mathbf b}\|_{L^2(\Omega)}^2
       + \|\nabla{\mathbf b}\|_{L^2(\Omega)} \|\nabla^2{\mathbf b}\|_{L^2(\Omega)}) \nonumber\\
& & + C \|{\mathbf b}\|_{L^2(\Omega)}
      \|\nabla{\mathbf b}\|_{L^2(\Omega)}(\|\nabla{\mathbf g}\|_{L^2(\Omega)}^2
       + \|\nabla{\mathbf g}\|_{L^2(\Omega)} \|\nabla^2{\mathbf g}\|_{L^2(\Omega)}+\|w\|_{L^4(\Omega)}^2) \nonumber\\
&\leq& \f{\nu}4 \|\nabla^2{\mathbf b}\|^2_{L^2(\Omega)} + \f{\nu}8\|\nabla^2{\mathbf g}\|_{L^2(\Omega)}^2
+ C\|{\mathbf u}\|_{L^2(\Omega)}^2 \|\nabla{\mathbf u}\|_{L^2(\Omega)}^2 \|\nabla{\mathbf b}\|_{L^2}^2 \nonumber\\
&&+ C\|{\mathbf u}\|_{L^2(\Omega)}\|\nabla{\mathbf u}\|_{L^2(\Omega)}\|\nabla{\mathbf b}\|_{L^2}^2 + C\|{\mathbf b}\|^2_{L^2(\Omega)} \|\nabla{\mathbf b}\|^2_{L^2(\Omega)} \|\nabla{\mathbf g}\|_{L^2}^2 \nonumber\\
&&+ C\|{\mathbf b}\|_{L^2(\Omega)} \|\nabla{\mathbf b}\|_{L^2(\Omega)} \|\nabla{\mathbf g}\|_{L^2}^2 + C\|{\mathbf b}\|_{L^2(\Omega)}\|\nabla{\mathbf b}\|_{L^2(\Omega)}\|w\|_{L^4}^2,\nonumber\\
&\leq& \f{\nu}4\|\nabla^2{\mathbf b}\|_{L^2}^2 + \f{\nu}8\|\nabla^2 {\mathbf g}\|_{L^2(\Omega)}^{2}
+ C(1 + \|{\mathbf u}\|_{L^2(\Omega)}^2\|\nabla{\mathbf u}\|_{L^2(\Omega)}^2+ \|{\mathbf b}\|_{L^2(\Omega)}^2\|\nabla{\mathbf b}\|_{L^2(\Omega)}^2)\times \nonumber \\
&& (\|\nabla{\mathbf g}\|_{L^2(\Omega)}^2 + \|\nabla{\mathbf b}\|_{L^2(\Omega)}^2 + \|w\|_{L^4}^2).
\end{eqnarray}

Multiplying \eqref{A106} by $ \nu $, \eqref{A106.1} by $ \mu + \chi $ and then adding them up, we obtain
\begin{eqnarray}
&&2C_1(\mu + \chi)\nu \f{d}{dt}(\|\nabla \mathbf g\|^2_{L^2(\Omega)}
   + \|\nabla \mathbf b\|^2_{L^2(\Omega)})
   + \f{3(\mu + \chi)\nu}4 \|\nabla^2 \mathbf g\|^2_{L^2(\Omega)}+ \f{(\mu + \chi)\nu}2 \|\nabla^2 \mathbf b\|^2_{L^2(\Omega)}\nonumber\\
  && + C_1 \nu \|\p_t{\mathbf g}\|_{L^2}^2
     + C_1(\mu + \chi)\|\p_t{\mathbf b}\|_{L^2}^2\nonumber\\
&\leq& C(1+\|{\mathbf u}\|_{L^2}^2 \|\nabla{\mathbf u}\|_{L^2}^2
    + \|{\mathbf b}\|_{L^2}^2\|\nabla{\mathbf b}\|_{L^2}^2)
   (\|\nabla{\mathbf g}\|_{L^2}^2 + \|\nabla{\mathbf b}\|_{L^2}^2 + \|w\|_{L^4}^2)+ C(\|{\mathbf u}\|_{L^2}^2+\|w\|_{L^2}^2).
\label{A100ub}
\end{eqnarray}

$ \bf Step ~3 $.
It is clear that \eqref{A100ub} is not a closed estimate since the bound of
$ \|w\|_{L^4(\Omega)} $ is unknown.
However, we found that the estimate of $ \|w\|_{L^4(\Omega)} $ can be bounded in turn by $ \|\nabla\mathbf g\|_{L^2(\Omega)} $ and $ \|\nabla^2\mathbf g\|_{L^2(\Omega)} $ because the equation of $ w $ in $\eqref{eq1}_2$ is not related to the magnetic field $ \mathbf b $.
Keeping this in mind, by multiplying $\eqref{eq1}_2$ with $|w|^2w$ and integrating on $\Omega$, we obtain
\ben
&&\f14\f{d}{dt}\|w\|_{L^4(\Omega)}^4 + 2\chi \|w\|_{L^4(\Omega)}^4=\chi \int_{\Omega}\nabla^{\perp}\cdot {\mathbf u}{|w|^2 w}\,dx \notag\\
&\leq& C\|\nabla{\mathbf u}\|_{L^4(\Omega)}\|w\|_{L^4(\Omega)}^3 \leq C(\|\nabla{\mathbf g}\|_{L^4(\Omega)}+\|\nabla{\mathbf v}\|_{L^4(\Omega)})\|w\|_{L^4(\Omega)}^3\notag\\
&\leq& C(\|\nabla{\mathbf g}\|_{L^2(\Omega)}+\|\nabla{\mathbf g}\|_{L^2(\Omega)}^{\f12}\|\nabla^2{\mathbf g}\|_{L^2(\Omega)}^{\f12}+\|w\|_{L^4(\Omega)})\|w\|_{L^4(\Omega)}^3. \notag
\een
Then by dividing $ \|w\|_{L^4(\Omega)}^2 $ on both sides and using Young's inequality, it yields
\ben
\f12\f{d}{dt}\|w\|_{L^4(\Omega)}^2 + 2\chi\|w\|_{L^4(\Omega)}^2\leq \f1{16C_1}\|\nabla^2{\mathbf g}\|_{L^2(\Omega)}^{2}+C(\|\nabla{\mathbf g}\|_{L^2(\Omega)}^2+\|w\|_{L^4(\Omega)}^2).
\label{A108}
\een

$ \bf Step ~4 $.
Multiplying \eqref{A108} by $ 4C_1(\mu + \chi)\nu $ and summing the resultant with \eqref{A100ub}, we finally obtain
\ben
&&2C_1 (\mu + \chi) \nu \f{d}{dt}(\|\nabla\mathbf g\|_{L^2(\Omega)}^2
  + \|\nabla\mathbf b\|_{L^2(\Omega)}^2 + \|w\|_{L^4(\Omega)}^2)
  + \f{(\mu+\chi)\nu}2 (\|\nabla^2 {\mathbf g}\|_{L^2(\Omega)}^2+ \|\nabla^2 {\mathbf b}\|_{L^2}^2)
  \notag\\
& &+ 8C_1 \chi (\mu + \chi)\nu \|w\|_{L^4}^2
   + C_1 \nu \|\p_t{\mathbf g}\|_{L^2}^2
   + C_1(\mu + \chi)\|\p_t{\mathbf b}\|_{L^2}^2
     \notag\\
&\leq& C(1+\|{\mathbf u}\|_{L^2}^2 \|\nabla{\mathbf u}\|_{L^2}^2
    + \|{\mathbf b}\|_{L^2}^2\|\nabla{\mathbf b}\|_{L^2}^2)
   (\|\nabla{\mathbf g}\|_{L^2}^2 + \|\nabla{\mathbf b}\|_{L^2}^2 + \|w\|_{L^4}^2)+C(\|{\mathbf u}\|_{L^2(\Omega)}^2+\|w\|_{L^2(\Omega)}^2),\notag
\label{A109}
\een
where $ C \gg C_1 $.
This, together with Gr\"{o}nwall's inequality and Lemma \ref{uL2} yields
\begin{eqnarray*}\label{A2}
&&\sup\limits_{0 \leq t \leq T}(\|\nabla{\mathbf g}\|^2_{L^2(\Omega)} + \|\nabla{\mathbf b}\|^2_{L^2(\Omega)} + \|w\|^2_{L^4(\Omega)}) +
\int_0^t(\|\p_t{\mathbf g}\|^2_{L^2(\Omega)}
+ \|\p_t{\mathbf b}\|^2_{L^2(\Omega)})\,dt   \\
&&+\int_0^t(\|\nabla^2{\mathbf g}\|^2_{L^2(\Omega)}+ \|\nabla^2{\mathbf b}\|^2_{L^2(\Omega)} + \|w\|^2_{L^4(\Omega)})\,dt \leq C,
\end{eqnarray*}
where $C = C(\Omega, T, \|{\mathbf u}_0\|_{H^1(\Omega)}, \|{\mathbf b}_0\|_{H^1(\Omega)}, \|w_0\|_{L^4(\Omega)})$. Thus, the proof of Lemma \ref{gH1} is finished.
\end{proof}
~\\
\subsection{The estimate of $ \|{\mathbf u}\|_{L^2(0,T;H^2(\Omega))} $}
~\\
\label{section3.3}

To prove the global existence of weak solutions, we need the estimates of
$ \|{\mathbf u}\|_{L^2(0, T; H^2(\Omega))} $ and
$ \|{\mathbf b}\|_{L^2(0, T; H^2(\Omega))} $.
The estimate of
$ \|{\mathbf b}\|_{L^2(0, T; H^2(\Omega))} $ has been obtained in Lemma \ref{gH1}. For the estimate of
$ \|{\mathbf u}\|_{L^2(0, T; H^2(\Omega))}$, by Lemma \ref{stokes}, it follows that
\begin{eqnarray}\label{uH2}
\|{\mathbf u}\|_{H^2(\Omega)}
\leq C_1(\| \p_t\mathbf u \|_{L^2(\Omega)} + \|\mathbf u \cdot \nabla\mathbf u\|_{L^2(\Omega)} + \|\mathbf b \cdot \nabla\mathbf b\|_{L^2(\Omega)} + \|\nabla w\|_{L^2(\Omega)}).
\end{eqnarray}
Then by H\"{o}lder's inequality, Corollary \ref{C1} and Young's inequality,
we get
\begin{eqnarray}\label{uh21}
&&\|\mathbf u \cdot \nabla\mathbf u\|_{L^2(\Omega)}
\leq \|\mathbf u\|_{L^4(\Omega)} \|\nabla\mathbf u\|_{L^4(\Omega)} \nonumber\\
&\leq& C_2 \|\mathbf u\|_{L^2(\Omega)}^{\f12} \|\nabla\mathbf u\|_{L^2(\Omega)}^{\f12}(\|\nabla\mathbf u\|_{L^2(\Omega)}^{\f12} \|\nabla^2\mathbf u\|_{L^2(\Omega)}^{\f12} + \|\nabla\mathbf u\|_{L^2(\Omega)})
       \nonumber\\
&\leq& 2 C_2 \|\mathbf u\|_{H^1(\Omega)}^{\f32} \|\mathbf u\|_{H^2(\Omega)}^{\f12}\leq \f1{2C_1} \|\mathbf u\|_{H^2(\Omega)}
     + 2C_1 C_2^2 \|\mathbf u\|_{H^1(\Omega)}^3,
\end{eqnarray}
and similarly,
\begin{eqnarray}\label{uh22}
\|\mathbf b \cdot \nabla\mathbf b\|_{L^2(\Omega)}
\leq \f1{2C_1} \|\mathbf b\|_{H^2(\Omega)}
     + 2C_1 C_2^2 \|\mathbf b\|_{H^1(\Omega)}^3.
\end{eqnarray}
Plugging \eqref{uh21} and \eqref{uh22} into \eqref{uH2}, we get
\begin{eqnarray}\label{uH21}
&&\|\mathbf u\|_{H^2(\Omega)}\nonumber\\
&\leq& 2C_1 \|\p_t\mathbf u \|_{L^2(\Omega)} + 4 C_1^2 C_2^2(\|\mathbf u\|_{H^1(\Omega)}^3 + \|\mathbf b\|_{H^1(\Omega)}^3) + \|\mathbf b\|_{H^2(\Omega)} + 2C_1\|\nabla w\|_{L^2}  \nonumber\\
&\leq& C (\|\p_t\mathbf g \|_{L^2(\Omega)}+ \|\p_t w \|_{L^2(\Omega)})
   + C(\|\mathbf u\|_{H^1(\Omega)}^3 + \|\mathbf b\|_{H^1(\Omega)}^3) + \|\mathbf b\|_{H^2(\Omega)} + C\|\nabla w\|_{L^2(\Omega)}.
\end{eqnarray}
By applying Lemmas \ref{uL2}--\ref{gH1}, it holds that
\begin{eqnarray}\label{uH22}
\|\mathbf u\|_{L^2(0,T; H^2(\Omega))}
&\leq& C (\|\p_t\mathbf g \|_{L^2(0,T; L^2(\Omega))}+ \|\p_t w \|_{L^2(0,T; L^2(\Omega))}) + C \sqrt{T} \sup_{0 \leq t \leq T}\|\mathbf u\|_{H^1(\Omega)}^3    \nonumber\\
  && + C\sqrt{T} \sup_{0 \leq t \leq T}\|\mathbf b\|_{H^1(\Omega)}^3 + \|\mathbf b\|_{L^2(0,T; H^2(\Omega))} + C\|\nabla w\|_{L^2(0,T; L^2(\Omega))}\nonumber\\
&\leq& C(1 + \|\p_t w \|_{L^2(0,T; L^2(\Omega))}
       + \|\nabla w\|_{L^2(0,T; L^2(\Omega))}).
\end{eqnarray}

In the following, we focus on establishing estimates of
$ \|\p_t w\|_{L^2(\Omega)} $ and
$ \|\nabla w\|_{L^2(\Omega)} $.
In fact, we will see later that the estimate of
$ \|\p_t w\|_{L^2(\Omega)} $ can be bounded by
$ \|\mathbf u \|_{H^1(\Omega)} $ and $ \|\nabla w\|_{L^4(\Omega)} $.
Moreover, the bound of $ \|\nabla w\|_{L^4(\Omega)} $ can be controlled by $ \|w\|_{L^\infty(0, T;L^q(\Omega))} $ and $ \|\nabla^2 \mathbf g\|_{L^2(0, T;L^q(\Omega))} $.
To this end, we first establish the estimate of
$ \|w\|_{L^\infty(0, T;L^q(\Omega))} $.

\begin{lemma}\label{w-Lp}
In addition to the conditions in Lemma \ref{gH1}, if we further assume
$w_0 \in L^q (\Omega)$ for any $2\leq q \leq \infty$, then it holds that
\begin{equation*}
\sup\limits_{0\leq t\leq T}\|w\|_{L^q(\Omega)} \leq C,
\end{equation*}
where the constant $C$ depends only on $\Omega$, $T$, $\|{\mathbf u}_0\|_{H^1(\Omega)}$, $\|{\mathbf b}_0\|_{H^1(\Omega)}$ and $\|{w}_0\|_{L^{q}(\Omega)}$.
\end{lemma}
\begin{proof}
The proof of Lemma \ref{w-Lp} is similar with the estimate of $ \|w\|_{L^4(\Omega)} $ in Step 3 of Lemma \ref{gH1}, we give only a sketch here.
Multiplying $ \eqref{eq1}_2 $ by $ |w|^{q-2} w $ and integrating on $ \Omega $,
it follows that the estimate of $ \|w\|_{L^q(\Omega)} $ can be bounded by $ \|{\mathbf g}\|_{L^1(0, T; H^2(\Omega))} $, which is the consequence of Lemma \ref{uL2} and Lemma \ref{gH1}.
\end{proof}
\vskip 2mm

In the following lemma, we establish the bound of
$ \|{\mathbf g}\|_{L^2(0, T; W^{2,q}(\Omega))} $.

\begin{lemma}\label{gW2q}
Under the assumptions of  Lemma \ref{w-Lp}, if in addition,
$\mathbf {u}_0 \in H^{2}(\Omega)$ and $w_0 \in H^{1}(\Omega)$, then
for any $ 2 \leq q < \infty $, it holds that
\begin{equation*}
  \int_0^T \|\nabla^2{\mathbf g}\|^2_{L^q(\Omega)} dt \leq C,
\end{equation*}
where the constant $C$ depends only on $\Omega$, $T$, $\|{\mathbf u}_0\|_{H^2(\Omega)}$, $\|{\mathbf b}_0\|_{H^1(\Omega)}$ and $\|{w}_0\|_{H^{1}(\Omega)}$.
\end{lemma}

\begin{proof}
First, by applying Lemma \ref{timestokes} to \eqref{mathg} and Lemma \ref{stokes}, it is clear that
\ben
&&\|\nabla^2\mathbf g\|_{L^2(0,T;L^q(\Omega))}\leq C(\|\mathbf Q\|_{L^2(0,T;L^q(\Omega))}+\|{\mathbf g}_0\|_{H^{2}(\Omega)})\notag\\
&\leq&C(\|\mathbf Q\|_{L^2(0,T;L^q(\Omega))}+\|{\mathbf u}_0\|_{H^{2}(\Omega)}+\|{w}_0\|_{H^{1}(\Omega)}),
\label{A110}
\een
with
\ben
\|\mathbf Q\|_{L^2(0,T;L^q(\Omega))}
&\leq& \|{\mathbf b}\cdot\nabla{\mathbf b}\|_{L^2(0,T;L^q(\Omega))}+ \|{\mathbf u}\cdot\nabla{\mathbf u}\|_{L^2(0,T;L^q(\Omega))}+ C\|A^{-1}\nabla^{\perp}({\mathbf u} \cdot
        \nabla{w})\|_{L^2(0,T;L^q(\Omega))}\notag\\
& &+ C\|A^{-1}\nabla^{\perp}(\nabla^{\perp}\cdot
      {\mathbf u})\|_{L^2(0,T;L^q)} + C\|{\mathbf v}\|_{L^2(0,T;L^q(\Omega))}=\sum\limits_{i=1}^5 I_i.
\label{A111}
\een

Next, for the five terms $ I_i $, $ i=1,2, \dots, 5 $,
by employing H\"{o}lder inequality, Sobolev embedding inequalities, Lemmas \ref{uL2}--\ref{w-Lp}, \eqref{keyestimates} and the equality ${\mathbf u}\cdot\nabla w=\nabla\cdot({\mathbf u}w)$, we have
\begin{eqnarray}
\label{A111.1}
I_1
&\leq& \|{\mathbf b}\|_{L^\infty(0,T;L^{2q}(\Omega))}
       \|\nabla{\mathbf b}\|_{L^2(0,T;L^{2q}(\Omega))}\leq C \|{\mathbf b}\|_{L^\infty(0,T;H^{1}(\Omega))}
        \|{\mathbf b}\|_{L^2(0,T;H^{2}(\Omega))}\leq C,\notag\\
\label{A112}
I_2
&\leq& \|{\mathbf u}\|_{L^\infty(0,T;L^{2q}(\Omega))}\|\nabla{\mathbf u}\|_{L^2(0,T;L^{2q}(\Omega))}\notag\\
&\leq& C\|{\mathbf u}\|_{L^\infty(0,T;H^{1}(\Omega))}(\|\nabla{\mathbf g}\|_{L^2(0,T;L^{2q}(\Omega))}+\|\nabla{\mathbf v}\|_{L^2(0,T;L^{2q}(\Omega))})\notag\\
&\leq& C(\|{\mathbf u}\|_{L^\infty(0,T;L^{2}(\Omega))}+\|\nabla{\mathbf u}\|_{L^\infty(0,T;L^{2}(\Omega))})(\|{\mathbf g}\|_{L^2(0,T;H^{2}(\Omega))}+\|w\|_{L^2(0,T;L^{2q}(\Omega))})\notag\\
&\leq& C(1 + \|\nabla{\mathbf g}\|_{L^\infty(0,T;L^{2}(\Omega))}+\|\nabla{\mathbf v}\|_{L^\infty(0,T;L^{2}(\Omega))})(\|{\mathbf g}\|_{L^2(0,T;L^{2}(\Omega))}+\|\Delta{\mathbf g}\|_{L^2(0,T;L^{2}(\Omega))}+1)\notag\\
&\leq& C(1 + \|w\|_{L^\infty(0,T;L^{2}(\Omega))})(\|{\mathbf u}\|_{L^2(0,T;L^{2}(\Omega))}+\|w\|_{L^2(0,T;L^{2}(\Omega))}+ 1)\leq C, \notag\\
\label{A113}
\sum_{i=3}^5 I_i
&\leq& C \|A^{-1}\nabla^{\perp}\nabla\cdot ({\mathbf u}{w})\|_{L^2(0,T;L^{q}(\Omega))}
   + C \|A^{-1}\nabla^{\perp}(\nabla^{\perp}\cdot
   {\mathbf u})\|_{L^2(0,T;L^{q}(\Omega))}+ C \|{\mathbf v}\|_{L^2(0,T;L^{q}(\Omega))}  \notag\\
&\leq& C\|{\mathbf u}{w}\|_{L^2(0,T;L^{q}(\Omega))}
   + C\|{\mathbf u}\|_{L^2(0,T;L^{q}(\Omega))}
   + C\|w\|_{L^2(0,T;L^{q}(\Omega))}\notag\\
&\leq& C\|{\mathbf u}\|_{L^2(0,T;H^{1}(\Omega))}\|{w}\|_{L^\infty(0,T;L^{2q}(\Omega))} + C\|{\mathbf u}\|_{L^2(0,T;H^{1}(\Omega))} + C\|w\|_{L^2(0,T;L^{q}(\Omega))}\leq C. \notag
\end{eqnarray}
By summing up the above estimates and using Lemmas \ref{uL2}--\ref{w-Lp}, we finally obtain $\|\mathbf g\|_{L^2(0,T;W^{2,q}(\Omega))}\leq C$.
\end{proof}

\vskip .1in

With the help of Lemmas \ref{w-Lp} and \ref{gW2q}, we now return to establish the estimate of $ \|\nabla w \|_{L^q(\Omega)} $.

\begin{lemma}\label{nablaw-Lp}
In addition to the conditions in Lemma \ref{gW2q}, we further assume
$ \nabla w_0 \in L^{q}(\Omega) $ for any $ 2 < q < \infty $, then it holds that
\begin{equation*}
\sup\limits_{0\leq t\leq T} \|\nabla w\|_{L^q(\Omega)} \leq C,
\end{equation*}
where the constant $ C $ depends only on $\Omega$, $ T $, $ \|{\mathbf u}_0\|_{H^{2}(\Omega)} $, $ \|{\mathbf b}_0\|_{H^{1}(\Omega)} $, $ \|{w}_0\|_{H^{1}(\Omega)} $ and
$ \|\nabla{w}_0\|_{L^{q}(\Omega)} $.
\end{lemma}
\begin{proof}
Differentiating $ \eqref{eq1}_2 $ with respect to spatial variable $ x_i $, $ i=1, 2 $, we have
\begin{eqnarray}\label{piw}
\p_i w_t + 2 \chi \p_iw + \mathbf u \cdot \nabla \p_iw + \p_i \mathbf u \cdot \nabla w
= \chi \p_i \nabla^\bot \cdot \mathbf u, \quad i=1,2.
\end{eqnarray}
Taking the inner product of \eqref{piw} with $ |\p_i w|^{q-2}\p_i w $ for any $ 2 \leq q < \infty $
and summing over $ i =1, 2 $, we get
\begin{eqnarray*}
\f1{q} \f{d}{dt} \|\nabla w\|_{L^q(\Omega)}^q + 2 \chi \|\nabla w\|_{L^q(\Omega)}^q
\leq \|\nabla \mathbf u\|_{L^\infty(\Omega)} \|\nabla w\|_{L^q(\Omega)}^q + \chi \|\nabla^2\mathbf u\|_{L^q(\Omega)} \|\nabla w\|_{L^q(\Omega)}^{q-1},
\end{eqnarray*}
and thus we have
\begin{eqnarray*}
\f{d}{dt} \|\nabla w\|_{L^q(\Omega)} + 2 \chi \|\nabla w\|_{L^q(\Omega)}
\leq \|\nabla \mathbf u\|_{L^\infty(\Omega)} \|\nabla w\|_{L^q(\Omega)}
     + \chi \|\nabla^2\mathbf u\|_{L^q(\Omega)}.
\end{eqnarray*}
Then, by using the Sobolev embedding inequalities, \eqref{keyestimates1} and \eqref{keyestimates}, it follows that
\begin{eqnarray*}
&&\f{d}{dt} \|\nabla w\|_{L^q(\Omega)} + 2 \chi \|\nabla w\|_{L^q(\Omega)}\\
&\leq& \|\nabla \mathbf g\|_{L^\infty(\Omega)} \|\nabla w\|_{L^q(\Omega)} + \|\nabla \mathbf v\|_{L^\infty(\Omega)} \|\nabla w\|_{L^q(\Omega)}
     + \chi \|\nabla^2\mathbf g\|_{L^q(\Omega)} + \chi \|\nabla^2\mathbf v\|_{L^q(\Omega)}\\
&\leq& C \|\mathbf g\|_{W^{2,q}(\Omega)} \|\nabla w\|_{L^q(\Omega)} + C(1+ \|w\|_{L^\infty(\Omega)}) \ln(e+\|\nabla w\|_{L^q(\Omega)}) \|\nabla w\|_{L^q(\Omega)} + C \|\mathbf g\|_{W^{2,q}(\Omega)}+ C \|\nabla w\|_{L^q(\Omega)}  \\
&\leq& C (1+\|w\|_{L^\infty(\Omega)})(1+\|\mathbf g\|_{W^{2,q}(\Omega)})(e+\|\nabla w\|_{L^q(\Omega)})\ln(e+\|\nabla w\|_{L^q(\Omega)}),
\end{eqnarray*}
for any $ 2 < q < \infty $.
This, together with Gr\"{o}nwall's inequality, Lemmas \ref{w-Lp}--\ref{gW2q}, yields
\begin{eqnarray*}
\sup\limits_{0 \leq t \leq T} \|\nabla w\|_{L^q(\Omega)} + 2 \chi \int_0^T \|\nabla w\|_{L^q(\Omega)} \, dt \leq C,
\end{eqnarray*}
where $ C $ depends only on $\Omega$, $ T $, $ \|{\mathbf u}_0\|_{H^{2}(\Omega)} $, $ \|{\mathbf b}_0\|_{H^{1}(\Omega)} $, $ \|{w}_0\|_{H^{1}(\Omega)} $ and
$ \|\nabla{w}_0\|_{L^{q}(\Omega)} $.
\end{proof}
\vskip 2mm

\begin{lemma}\label{ptw}
Under the assumptions of Lemma \ref{nablaw-Lp}, it holds that
\begin{equation*}
 \sup\limits_{0 \leq t \leq T}\|\p_t w\|^2_{L^2(\Omega)}\leq C,
\end{equation*}
where the constant $ C $ depends only on $\Omega$, $ T $, $ \|{\mathbf u}_0\|_{H^{2}(\Omega)} $, $ \|{\mathbf b}_0\|_{H^{1}(\Omega)} $, $ \|{w}_0\|_{H^{1}(\Omega)} $ and
$ \|\nabla{w}_0\|_{L^{4}(\Omega)} $.
\end{lemma}
\begin{proof}
 Multiplying the equation $ \eqref{eq1}_2 $ with $ \p_t w $ and integrating on
$ \Omega $, we get
\begin{eqnarray*}
\|\p_t w\|_{L^2(\Omega)}^2
&=& - \int_{\Omega}\mathbf u \cdot \nabla w \p_tw \,dx
    - 2 \chi \int_{\Omega} w \p_tw \,dx
    + \chi \int_{\Omega} \nabla^{\bot} \cdot \mathbf u \p_tw \,dx   \nonumber\\
&\leq& \f12 \|\p_t w\|_{L^2(\Omega)}^2+ C \|\mathbf u \cdot \nabla w\|_{L^2(\Omega)}^2
       + C \|w\|_{L^2(\Omega)}^2
       + C\|\nabla\mathbf u\|_{L^2(\Omega)}^2  \nonumber\\
&\leq& \f12 \|\p_t w\|_{L^2(\Omega)}^2+ C \|\mathbf u \|_{H^1(\Omega)}^2 \|\nabla w\|_{L^4(\Omega)}^2
      + C \|w\|_{L^2(\Omega)}^2
      + C\|\mathbf u\|_{H^1(\Omega)}^2,
\end{eqnarray*}
which implies, after applying Lemma \ref{gH1} and Lemma \ref{nablaw-Lp}, that
$ \sup\limits_{0 \leq t \leq T}\|\p_t w\|^2_{L^2(\Omega)}\leq C $.
\end{proof}
\vskip 1.5mm

By using $ \mathbf u = \mathbf g + \mathbf v $, \eqref{keyestimates}, Lemma \ref{gH1} and Lemma \ref{ptw}, we get
$$ \|\p_t\mathbf u\|_{L^2(0, T; L^2(\Omega))} \leq C (\|\p_t\mathbf g\|_{L^2(0, T; L^2(\Omega))} + \|\p_t w\|_{L^2(0, T; L^2(\Omega))}) \leq C. $$

We now turn to estimate $ \|\mathbf u\|_{L^2(0, T; H^2(\Omega))} $.
By recalling inequality \eqref{uH22}, and using H\"{o}lder's inequality, Lemmas \ref{nablaw-Lp}--\ref{ptw},
we have
\begin{eqnarray}
&&\|\mathbf u\|_{L^2 (0, T; H^2(\Omega))}\leq C(1 + \|\p_t w \|_{L^2(0,T; L^2(\Omega))}
       + \|\nabla w\|_{L^2(0,T; L^2(\Omega))})\notag\\
&\leq& C (1 + \sqrt{T} \sup_{0 \leq t \leq T}\|\p_t w \|_{L^2(\Omega)}
       + \sqrt{T} \sup_{0 \leq t \leq T}\|\nabla w\|_{L^4(\Omega)})\leq C.\label{ul2h2}
\end{eqnarray}
In addition, with the help of Lemma \ref{gW2q} and Lemma \ref{nablaw-Lp}, we have
$$ \|\nabla^2\mathbf u\|_{L^2(0, T; L^4(\Omega))}
\leq C (\| \nabla^2 \mathbf g \|_{L^2(0, T; L^4(\Omega))}
+ \| \nabla w \|_{L^2(0, T; L^4(\Omega))}) \leq C. $$
Furthermore, we have the estimate of $ \|\p_t w\|_{L^4(\Omega)} $ in the next Lemma, which is used to guarantee the continuity of
the micro-rotation field at the initial time.
\begin{lemma}\label{ptw4}
Under the assumptions in Lemma \ref{ptw}, it holds that
\begin{equation*}
\int_0^T \|\p_t w\|_{L^4(\Omega)}^2 \,dt \leq C,
\end{equation*}
where $ C $ depends only on $\Omega$, $ T $, $ \|{\mathbf u}_0\|_{H^{2}(\Omega)} $, $ \|{\mathbf b}_0\|_{H^{1}(\Omega)} $, $ \|{w}_0\|_{H^{1}(\Omega)} $ and
$ \|\nabla{w}_0\|_{L^{4}(\Omega)} $.
\end{lemma}
\begin{proof}
We multiply the equation
$ \eqref{eq1}_2 $ with $ |\p_t w|^2 \p_t w $, integrate on $ \Omega $ and use the Sobolev embedding inequalities to get
\begin{eqnarray*}
\|\p_t w\|_{L^4(\Omega)}^2
&\leq & C \|\mathbf u \cdot \nabla w\|_{L^4(\Omega)}^2
       + C \|w\|_{L^4(\Omega)}^2
       + C\|\nabla\mathbf u\|_{L^4(\Omega)}^2  \\
&\leq& C \|\mathbf u \|_{H^2(\Omega)}^2 \|\nabla w\|_{L^4(\Omega)}^2
      + C \|w\|_{L^4(\Omega)}^2
      + C\|\mathbf u\|_{H^2(\Omega)}^2.
\end{eqnarray*}
By using \eqref{ul2h2} and Lemma \ref{nablaw-Lp},
we obtain $ \|\p_t w\|^2_{L^2(0, T; L^4(\Omega))} \leq C $.
\end{proof}
\vskip 2mm

Therefore, by Lemmas \ref{uL2}--\ref{ptw4}, we obtain the desired estimates in Proposition \ref{uH1}.
~\\
\subsection{Existence of weak solutions}
~\\
\label{section3.4}

In this subsection, we establish the following global existence of weak solutions to the system \eqref{eq1}--\eqref{eq20}.

\begin{theorem}\label{T2}
Let $ \Omega \subset \R^2 $ be a bounded domain with smooth boundary.
Suppose that the initial data ${\mathbf u}_0 \in H_0^1(\Omega) \cap H^{2}(\Omega), w_0\in W^{1,4}(\Omega), {\mathbf b}_0 \in H_0^1(\Omega)$, then there exist global in time weak solutions
$ ({\mathbf u}, w, {\mathbf b}) $ of system \eqref{eq1}--\eqref{eq20} such that for any $ T > 0 $,
\begin{eqnarray*}
& & {\mathbf u} \in L^\infty(0, T; H_0^1(\Omega))
     \cap L^2(0, T; W^{2,4}(\Omega)),\\
& & w \in L^\infty(0, T; W^{1,4}(\Omega)),\\     
& & {\mathbf b} \in L^\infty(0, T; H_0^1(\Omega))
     \cap L^2(0, T; H^2(\Omega)).
\end{eqnarray*}
\end{theorem}

\begin{proof}
The proof of the existence of weak solutions is based on Schauder's fixed point theorem.
To define the functional setting, we fix any $ T > 0 $ while the constant $ R_0 $ will be specified later. For notational convenience, we write
\begin{equation*}
  X \equiv C(0, T; L^4(\Omega)),
\end{equation*}
with $ \| f \|_X^2 = \| f \|_{C(0, T; L^4(\Omega))}^2 $,
and define $D = \{ f \in X | \|  f \|_X \leq R_0 \}.$ Clearly, $ D \subset X $ is closed and convex.

\vskip .1in
$ \bf Setting\,\,up $.
We fix $ \epsilon \in (0, 1) $ and define a continuous map on $ D $.
For any $ f \in D $, we regularize it and also the initial data $ (\mathbf u_0, w_0, \mathbf b_0) $ via the standard mollifying process,
$f^\epsilon = \rho^\epsilon * f,\,\mathbf u_0^\epsilon = \rho^\epsilon * \mathbf u_0,\,w_0^\epsilon = \rho^\epsilon * w_0,\,\mathbf b_0^\epsilon = \rho^\epsilon * \mathbf b_0,$
where $ \rho^\epsilon $ is a standard mollifier. Then it holds that
\begin{eqnarray}
\label{fepsilon}
&&f^\epsilon \in C(0, T; C_0^\infty(\Omega)),\quad
  \|f^\epsilon\|_{C(0, T; L^4(\Omega))}
     \leq \| f\|_{C(0, T; L^4(\Omega))}, \nonumber\\
&&\mathbf u_0^\epsilon \in C_0^\infty(\Omega),\quad
  \nabla \cdot \mathbf u_0^\epsilon = 0, \quad
  \|\mathbf u_0^\epsilon - \mathbf u_0\|_{H^1(\Omega)} < \epsilon,\nonumber\\
&&w_0^\epsilon \in C_0^\infty(\Omega),\quad
  \|w_0^\epsilon - w_0\|_{L^4({\Omega})} < \epsilon,\nonumber\\
&&\mathbf b_0^\epsilon \in C_0^\infty(\Omega),\quad
  \nabla \cdot \mathbf b_0^\epsilon = 0, \quad
  \|\mathbf b_0^\epsilon - \mathbf b_0\|_{H^1(\Omega)} < \epsilon \nonumber.
\end{eqnarray}

In fact, for the following 2D MHD system with smooth external forcing
$ \chi\nabla^\bot f^\epsilon $ and smooth initial data
$ (\mathbf u_0^\epsilon, \mathbf b_0^\epsilon)$
\begin{eqnarray}\label{ub1}
\left\{\begin{array}{ll}
  \p_t \mathbf u + \mathbf u \cdot \nabla \mathbf u - (\mu + \chi) \Delta \mathbf u  + \nabla p = \mathbf b \cdot \nabla \mathbf b - \chi \nabla^\bot f^\epsilon, \\
  \p_t \mathbf b + \mathbf u \cdot \nabla \mathbf b -\nu \Delta \mathbf b = \mathbf b \cdot \nabla \mathbf u,   \\
  \nabla \cdot \mathbf u = \nabla \cdot \mathbf b = 0,\quad \mathbf u|_{\p\Omega} = \mathbf b|_{\p\Omega} = 0,   \\
  (\mathbf u, \mathbf b)(x, 0) = (\mathbf u_0^\epsilon, \mathbf b_0^\epsilon)(x),
\end{array}\right.
\end{eqnarray}
it is proved \cite{DuLions} that such a system with $  f^\epsilon = 0 $ has a unique solution $ (\mathbf u^\epsilon, \mathbf b^\epsilon) $.
In fact, with minor modifications of the proof in \cite{DuLions}, it can be shown that the system \eqref{ub1} also has a unique solution $ (\mathbf u^\epsilon, \mathbf b^\epsilon) $ for any external forcing $ f^\epsilon \in C(0,T; C_0^\infty(\Omega)) $.

We then solve the following linearized equation with the smooth initial data $ w_0^\epsilon $
\begin{eqnarray}\label{w1}
\left\{\begin{array}{ll}
  \p_t w + \mathbf u^\epsilon \cdot \nabla w = -2\chi w + \chi\nabla^\bot \cdot\mathbf u^\epsilon,\\
  w(x, 0) = w_0^\epsilon(x),
\end{array}\right.
\end{eqnarray}
and denote the solution by $ w^\epsilon $.
This process allows us to define the map
\begin{equation}
  F^\epsilon(f) = w^\epsilon.
\end{equation}

In the following, we will prove that $ F^\epsilon $ satisfies the conditions of Schauder's fixed point theorem, that is, for any fixed $ \epsilon \in (0, 1) $,
$ F^\epsilon: D \rightarrow D $ is continuous and compact.
More precisely, one needs to show
\begin{itemize}
\item[(a)] $ \|w^\epsilon\|_X \leq R_0 $ for any $ f \in D $;
\item[(b)] $ \|w^{\epsilon}\|_{C(0,T; W^{1,4}(\Omega))} \leq C $ for any $ f \in D $;
\item[(c)] For any $ \eta >0 $, there exists $ \delta = \delta(\eta) > 0 $ such that for any $ f_1, f_2 \in D $
with $ \| f_1 - f_2 \|_X < \delta $, it holds that $ \|F^{\epsilon}({f}_1) - F^{\epsilon} ({f}_2)\|_X < \eta $.
\end{itemize}

\vskip .1in
We verify (a) first. Taking inner products of the first two equations of \eqref{ub1} with $ \mathbf u, \mathbf b $ respectively, and using $ -\chi \int_\Omega \nabla^\bot f^\epsilon \cdot \mathbf u \, dx = \chi \int_\Omega (\nabla^\bot \cdot \mathbf u) f^\epsilon \, dx $, we obtain
\begin{eqnarray*}
&&\f12 \f{d}{dt} (\|\mathbf u\|_{L^2(\Omega)}^2 + \|\mathbf b\|_{L^2(\Omega)}^2) + (\mu + \chi) \|\nabla \mathbf u\|_{L^2(\Omega)}^2 + \nu \|\nabla \mathbf b \|_{L^2(\Omega)}^2 \\
&=& \chi \int_\Omega (\nabla^\bot \cdot \mathbf u) f^\epsilon \, dx \leq \f{\mu+\chi}2 \|\nabla\mathbf u\|_{L^2(\Omega)}^2 + C \|f^\epsilon\|_{L^2(\Omega)}^2,
\end{eqnarray*}
and thus after integration over $ [0, T] $, we obtain
\begin{eqnarray} \label{ul2T}
&&\|\mathbf u^\epsilon\|_{L^2(\Omega)}^2 + \|\mathbf b^\epsilon\|_{L^2(\Omega)}^2 + (\mu + \chi) \int_0^T \| \nabla \mathbf u^\epsilon\|_{L^2(\Omega)}^2 \, dt + 2\nu \int_0^T \| \nabla \mathbf b^\epsilon \|_{L^2(\Omega)}^2 \, dt  \nonumber\\
&\leq& \|\mathbf u_0^\epsilon \|_{L^2(\Omega)}^2 + \|\mathbf b_0^\epsilon \|_{L^2}^2
      + C \int_0^T \| f^\epsilon \|_{L^2}^2\, dt \leq \| \mathbf u_0\|_{L^2(\Omega)}^2 + \| \mathbf b_0 \|_{L^2(\Omega)}^2
    + CT.
\end{eqnarray}

Then, taking inner products of the first two equations of \eqref{ub1} with
$ -\Delta \mathbf u $, $ -\Delta \mathbf b $ respectively, we get
\begin{eqnarray*}
&&\f12\f{d}{dt}(\|\nabla\mathbf u \|_{L^2(\Omega)}^2 + \|\nabla \mathbf b \|_{L^2(\Omega)}^2) + (\mu + \chi) \|\Delta \mathbf u\|_{L^2(\Omega)}^2 + \nu \|\Delta \mathbf b\|_{L^2(\Omega)}^2  \\
&=& \chi \int_\Omega \nabla^\bot f^\epsilon \cdot \Delta \mathbf u\, dx
   + 2 \int_{\Omega} \mathbf u \cdot \nabla \mathbf b \cdot \Delta \mathbf b \, dx,
\end{eqnarray*}
where we used the equalities 
$$\int_\Omega \mathbf u \cdot \nabla \mathbf u \cdot \Delta \mathbf u \, dx = 0,\,\,\, \int_\Omega \mathbf u \cdot \nabla \mathbf b \cdot \Delta \mathbf b \, dx
= - \int_\Omega \mathbf b \cdot \nabla \mathbf b \cdot \Delta \mathbf u \, dx- \int_\Omega \mathbf b \cdot \nabla \mathbf u \cdot \Delta \mathbf b \, dx.$$ 
By applying H\"{o}lder's inequality, Corollary \ref{C1} and Young's inequality, we get
\begin{eqnarray*}
&&\f12\f{d}{dt}(\|\nabla\mathbf u \|_{L^2(\Omega)}^2 + \|\nabla \mathbf b \|_{L^2(\Omega)}^2) + (\mu + \chi) \|\Delta \mathbf u\|_{L^2(\Omega)}^2 + \nu \|\Delta \mathbf b\|_{L^2(\Omega)}^2 \\
&\leq& \chi \|\nabla f^\epsilon\|_{L^2(\Omega)} \|\Delta \mathbf u\|_{L^2(\Omega)}
       + 2 \| \mathbf u \|_{L^\infty(\Omega)} \| \nabla \mathbf b \|_{L^2(\Omega)} \| \Delta \mathbf b \|_{L^2(\Omega)}\\
&\leq& \chi \|\nabla f^\epsilon\|_{L^2(\Omega)} \|\Delta \mathbf u\|_{L^2(\Omega)}
       + C (\| \mathbf u \|_{L^2(\Omega)}^{\f12} \| \Delta \mathbf u \|_{L^2(\Omega)}^{\f12} + \| \mathbf u \|_{L^2(\Omega)}) \| \nabla \mathbf b \|_{L^2(\Omega)} \| \Delta \mathbf b \|_{L^2(\Omega)}\\
&\leq& \f{\mu+\chi}2 \|\Delta \mathbf u\|_{L^2(\Omega)}^2 + \f{\nu}2 \|\Delta \mathbf b\|_{L^2(\Omega)}^2 + C\|\nabla f^\epsilon\|_{L^2(\Omega)}^2 + C \| \mathbf u \|_{L^2(\Omega)}^2 \| \nabla \mathbf b \|_{L^2(\Omega)}^4 + C \| \mathbf u \|_{L^2(\Omega)}^2 \| \nabla \mathbf b \|_{L^2(\Omega)}^2,
\end{eqnarray*}
which implies, after applying Gr\"{o}nwall's inequality and the inequality \eqref{ul2T}, that
\begin{eqnarray}\label{uR0}
&& \|\nabla \mathbf u^\epsilon \|_{L^2(\Omega)}^2 + \|\nabla\mathbf b^\epsilon \|_{L^2(\Omega)}^2 + (\mu + \chi) \int_0^T \|\Delta\mathbf u^\epsilon\|_{L^2}^2\, dt + \nu  \int_0^T \|\Delta\mathbf b^\epsilon\|_{L^2}^2\, dt \nonumber\\
&\leq& (\|\nabla \mathbf u_0^\epsilon \|_{L^2}^2 + \|\nabla \mathbf b_0^\epsilon \|_{L^2}^2 + C \int_0^T \|\nabla f^\epsilon\|_{L^2}^2 \, dt) \,
e^{\{C \int_0^T(\| \mathbf u^\epsilon \|_{L^2}^2 \| \nabla \mathbf b^\epsilon \|_{L^2}^2 + \| \mathbf u^\epsilon \|_{L^2}^2) \, dt\}}
\nonumber\\
&\leq& (\|\nabla \mathbf u_0\|_{L^2}^2
  + \|\nabla \mathbf b_0 \|_{L^2}^2
+ CT \| f^\epsilon \|_{C(0,T;C_0^\infty)}^2) \,
e^{\{C(\sup\limits_{0\leq t \leq T} \| \mathbf u^\epsilon \|_{L^2}^2) \int_0^T \| \nabla \mathbf b^\epsilon \|_{L^2}^2 \, dt + CT \sup\limits_{0\leq t \leq T} \| \mathbf u^\epsilon \|_{L^2}^2 \}} \nonumber\\
&\leq& (\|\nabla \mathbf u_0\|_{L^2}^2 + \|\nabla \mathbf b_0 \|_{L^2}^2 + CT )\,
e^{\{C (\| \mathbf u_0 \|_{L^2}^2 + \| \mathbf b_0 \|_{L^2}^2)^2+ C T^2\}}
\end{eqnarray}

In addition, by multiplying \eqref{w1} with $ |w|^2w $, using H\"{o}lder's inequality and Sobolev embedding inequalities, we obtain
\begin{eqnarray*}
&&\f14\f{d}{dt}\|w\|_{L^4(\Omega)}^4 + 2\chi \|w\|_{L^4(\Omega)}^4=\chi \int_{\Omega}\nabla^{\perp}\cdot {\mathbf u}^\epsilon {|w|^2 w}\,dx \\
&\leq& C \| \mathbf u^\epsilon\|_{W^{2,2}(\Omega)} \|w\|_{L^4(\Omega)}^3 \leq C \| \Delta \mathbf u^\epsilon\|_{L^2(\Omega)}\|w\|_{L^4(\Omega)}^3,
\end{eqnarray*}
which together with \eqref{uR0} imply
\begin{eqnarray}
&&\|w^\epsilon\|_{L^4(\Omega)} +
  2\chi \int_0^T\|w^\epsilon\|_{L^4(\Omega)}\,dt\leq \|w_0^\epsilon\|_{L^4(\Omega)} + C \int_0^T \| \Delta \mathbf u^\epsilon\|_{L^2(\Omega)}\, dt\notag\\
&\leq& \|w_0\|_{L^4(\Omega)} + C (\int_0^T \| \Delta \mathbf u^\epsilon\|_{L^2(\Omega)}^2 \, dt)^{\f12} (\int_0^T1^2 \, dt )^{\f12}  \notag\\
&\leq& \|w_0\|_{L^4(\Omega)} + C \sqrt{T}(\|\nabla \mathbf u_0\|_{L^2} + \|\nabla \mathbf b_0 \|_{L^2} + C \sqrt{T}) \,e^{\{C (\| \mathbf u_0 \|_{L^2}^2 + \| \mathbf b_0 \|_{L^2}^2)^2 + C T^2 \}}.\label{schaustep2}
\end{eqnarray}

To show that $ F^\epsilon $ maps $ D $ to $ D $,
it suffices to verify that the right hand side of \eqref{schaustep2} is bounded by $ R_0 $.
Then we obtain a condition for $ R_0 $, namely
\begin{eqnarray*}
&&\|w_0\|_{L^4(\Omega)} +
(C \sqrt{T}(\|\nabla \mathbf u_0\|_{L^2} + \|\nabla \mathbf b_0 \|_{L^2}) + C T) \,e^{\{C (\| \mathbf u_0 \|_{L^2}^2 + \| \mathbf b_0 \|_{L^2}^2)^2 + C T^2 \}} \,\leq \, R_0.
\end{eqnarray*}
Therefore, as long as $ T $ is sufficiently small, such that $ C \sqrt{T} \ll 1 $, $ C T \ll 1 $, $ CT^2 \ll 1 $, above inequality would hold. Similarly, we can show (b) and (c) under the condition that $T$ is sufficiently small. Schauder¡¯s fixed point theorem then allows us to conclude that the existence of a solution $({\mathbf u}^{\epsilon}, {w}^{\epsilon}, {\mathbf b}^{\epsilon})$ on a finite time interval $T$. The uniform estimates Proposition \ref{uH1} would allow us to pass the limit to obtain
a weak solution $ (\mathbf u, w, \mathbf b) $ to the system \eqref{eq1}--\eqref{eq20} in $ \Omega \times [0, T] $ for any $ T>0 $.
Due to the global bounds obtained in Proposition \ref{uH1}, the local solution produced by Schauder's fixed point theorem can be extended into a global solution via Picard type extension theorem. Hence this allows us to obtain the desired global weak solutions $ (\mathbf u, w, \mathbf b) $.
\end{proof}

\section{Global strong solution}
\label{strong solutions}
\setcounter{equation}{0}

This section is devoted to establishing the global strong solutions and proving the uniqueness of strong solutions to the system \eqref{eq1}--\eqref{eq20}.
~\\
\subsection{Existence of strong solutions}
~\\
\label{section 4.1}

In this subsection, we establish the global existence of strong solutions to the system \eqref{eq1}--\eqref{eq20} in the following theorem.

\begin{theorem}\label{T3}
Let $ \Omega \subset \R^2 $ be a bounded domain with smooth boundary.
Suppose that the initial data $ ({\mathbf u}_0, w_0, {\mathbf b}_0) $ satisfies
\begin{equation*}
{\mathbf u}_0 \in H_0^1(\Omega) \cap H^{2}(\Omega), \quad
 w_0\in W^{1,4}(\Omega),\quad
{\mathbf b}_0 \in H_0^1(\Omega),
\end{equation*}
then the system \eqref{eq1}--\eqref{eq20} admits strong solutions
$ ({\mathbf u}, w, {\mathbf b}) $ globally in time.

\end{theorem}

To prove Theorem \ref{T3}, we focus on establishing the $t$-weighted $ H^2 $ estimates as follows.
\begin{proposition} \label{ubtH2}
  Let $ \Omega \subset \R^2 $ be a bounded domain with smooth boundary.
Assume the initial data
$ {\mathbf u}_0 \in H_0^1(\Omega) \cap H^{2}(\Omega) $,
$ {\mathbf b}_0 \in H_0^1(\Omega) $,
$ w_0\in W^{1,4}(\Omega) $ and a triple $ ({\mathbf u}, w, {\mathbf b}) $ is the smooth solution of the system \eqref{eq1}--\eqref{eq20}.
Then it holds that
\begin{eqnarray*}
\sup\limits_{0 \leq t\leq T} \,
 t(\|\nabla^2\mathbf u\|_{L^2(\Omega)}^2 + \|\nabla^2\mathbf b\|^2_{L^2(\Omega)})
 + \int_0^T t(\|\nabla\p_t{\mathbf u}\|_{L^2(\Omega)}^2 + \|\nabla\p_t{\mathbf b}\|_{L^2(\Omega)}^2) \, dt
\leq C,
\end{eqnarray*}
where the constant $ C $ depends only on $\Omega$, $ T $, $ \|{\mathbf u}_0\|_{H^{2}(\Omega)} $, $ \|{\mathbf b}_0\|_{H^{1}(\Omega)} $, $ \|w_0\|_{W^{1,4}(\Omega)} $.
\end{proposition}

\begin{proof}
 $\bf Step ~1 $.
First, differentiating $ \eqref{mathg}_1 $ with respect to $ t $, and then multiplying it with $ \p_t{\mathbf g} $, we obtain
 \begin{eqnarray}\label{ut21}
&& \f12\f{d}{dt}\|\p_t{\mathbf g}\|_{L^2(\Omega)}^2
     + (\mu + \chi)\| \nabla \p_t {\mathbf g} \|_{L^2(\Omega)}^2=\int_{\Omega}\p_t{\mathbf Q} \cdot \p_t{\mathbf g} dt   \nonumber\\
&=& - \int_{\Omega} \p_t{\mathbf u} \cdot \nabla{\mathbf u} \cdot
    \p_t{\mathbf g} \,dx - \int_{\Omega} {\mathbf u} \cdot \nabla\p_t{\mathbf u} \cdot \p_t{\mathbf g} \,dx +  \int_{\Omega} \p_t{\mathbf b} \cdot \nabla{\mathbf b} \cdot \p_t{\mathbf g} \,dx \nonumber\\
 && +  \int_{\Omega} {\mathbf b}\cdot\nabla\p_t{\mathbf b}\cdot\p_t{\mathbf g}\,dx
    -\f{\chi}{\mu + \chi} \int_{\Omega} A^{-1}\nabla^{\bot}\nabla\cdot(w\p_t{\mathbf u} + {\mathbf u}\p_tw)\cdot\p_t{\mathbf g} \,dx \nonumber\\
 && + \f{\chi^2}{\mu + \chi} \int_{\Omega} A^{-1}\nabla^{\bot}(\nabla^{\bot}\cdot\p_t{\mathbf u})\cdot\p_t{\mathbf g} \,dx
 + 2\chi \int_{\Omega}\p_t{\mathbf v}\cdot\p_t{\mathbf g} \,dx= \sum_{i=1}^7 I_i.
 \end{eqnarray}

Then we will estimate the seven terms one by one. By applying H\"{o}lder inequality, Corollary \ref{C1}, Young's inequality and (\ref{keyestimates}), it holds that
\begin{eqnarray*}
\label{I1}
I_1
&\leq&  \|\p_t{\mathbf u}\|_{L^4(\Omega)} \|\nabla{\mathbf u}\|_{L^2(\Omega)}
     \|\p_t{\mathbf g}\|_{L^4(\Omega)}    \nonumber\\
&\leq& C \|\p_t{\mathbf u}\|_{L^2(\Omega)}^{\f12}
       \|\nabla\p_t{\mathbf u}\|_{L^2(\Omega)}^{\f12} \|\nabla{\mathbf u}\|_{L^2(\Omega)} \|\p_t{\mathbf g}\|_{L^2(\Omega)}^{\f12} \|\nabla\p_t{\mathbf g}\|_{L^2(\Omega)}^{\f12}  \nonumber\\
&\leq& C \|\p_t{\mathbf u}\|_{L^2(\Omega)}
       \|\nabla\p_t{\mathbf u}\|_{L^2(\Omega)}
       \|\nabla{\mathbf u}\|_{L^2(\Omega)}^2
       + C \|\p_t{\mathbf g}\|_{L^2(\Omega)} \|\nabla\p_t{\mathbf g}\|_{L^2(\Omega)}   \nonumber\\
&\leq& C(\|\p_t{\mathbf g}\|_{L^2(\Omega)} + \|\p_tw\|_{L^2(\Omega)})
        (\|\nabla\p_t{\mathbf g}\|_{L^2(\Omega)} + \|\p_tw\|_{L^2(\Omega)}) \|\nabla{\mathbf u}\|_{L^2(\Omega)}^2  \nonumber\\ 
        &&+ C \|\p_t{\mathbf g}\|_{L^2(\Omega)} \|\nabla\p_t{\mathbf g}\|_{L^2(\Omega)}   \nonumber\\
&\leq& \f{\mu + \chi}{14} \|\nabla\p_t{\mathbf g}\|_{L^2(\Omega)}^2
      +C \|\p_t{\mathbf g}\|_{L^2(\Omega)}^2 \|\nabla {\mathbf u}\|_{L^2(\Omega)}^4 + C\|\p_tw\|_{L^2(\Omega)}^2 \nonumber\\
     && + C \|\p_tw\|_{L^2(\Omega)}^2 \|\nabla{\mathbf u}\|_{L^2(\Omega)}^4
        + C \|\p_tw\|_{L^2(\Omega)}^2 \|\nabla{\mathbf u}\|_{L^2(\Omega)}^2
        + C \|\p_t{\mathbf g}\|_{L^2(\Omega)}^2,\\
I_2
&\leq&  \|\mathbf u\|_{L^4(\Omega)} \|\nabla\p_t{\mathbf u}\|_{L^2(\Omega)}
       \|\p_t{\mathbf g}\|_{L^4(\Omega)}    \nonumber\\
&\leq& C \|\mathbf u\|_{L^2(\Omega)}^{\f12}
       \|\nabla {\mathbf u}\|_{L^2(\Omega)}^{\f12}
       (\|\nabla\p_t{\mathbf g}\|_{L^2(\Omega)} + \|\p_t w\|_{L^2(\Omega)})
       \|\p_t{\mathbf g}\|_{L^2(\Omega)}^{\f12}
       \|\nabla \p_t{\mathbf g}\|_{L^2(\Omega)}^{\f12} \nonumber\\
&\leq& \f{\mu + \chi}{14} \|\nabla\p_t{\mathbf g}\|_{L^2(\Omega)}^2
      + C \|\mathbf u\|_{L^2(\Omega)}^2 \|\nabla {\mathbf u}\|_{L^2(\Omega)}^2 \|\p_t{\mathbf g}\|_{L^2(\Omega)}^2 + C \|\p_tw\|_{L^2(\Omega)}^2, \\
I_3 + I_4
&\leq&  \|\p_t{\mathbf b}\|_{L^4(\Omega)} \|\nabla {\mathbf b}\|_{L^2(\Omega)}
         \|\p_t{\mathbf g}\|_{L^4(\Omega)}
        + \|\mathbf b\|_{L^4(\Omega)} \|\nabla\p_t{\mathbf b}\|_{L^2(\Omega)} \|\p_t{\mathbf g}\|_{L^4(\Omega)}   \nonumber\\
&\leq& C \|\p_t{\mathbf b}\|_{L^2(\Omega)}^{\f12}
       \|\nabla\p_t{\mathbf b}\|_{L^2(\Omega)}^{\f12}
       \|\nabla{\mathbf b}\|_{L^2(\Omega)}
       \|\p_t{\mathbf g}\|_{L^2(\Omega)}^{\f12}
       \|\nabla\p_t{\mathbf g}\|_{L^2(\Omega)}^{\f12}  \nonumber\\
     &&+ C \|\mathbf b\|_{L^2(\Omega)}^{\f12}
       \|\nabla {\mathbf b}\|_{L^2(\Omega)}^{\f12}
       \|\nabla \p_t{\mathbf b}\|_{L^2(\Omega)}
       \|\p_t{\mathbf g}\|_{L^2(\Omega)}^{\f12}
       \|\nabla \p_t{\mathbf g}\|_{L^2(\Omega)}^{\f12} \nonumber\\
&\leq&  \f{\mu + \chi}7 \|\nabla \p_t{\mathbf g}\|_{L^2(\Omega)}^2
       + \f{\nu}5 \|\nabla \p_t{\mathbf b}\|_{L^2(\Omega)}^2
       + C \|\p_t{\mathbf b}\|_{L^2(\Omega)}^2
           \|\nabla{\mathbf b}\|_{L^2(\Omega)}^4\nonumber\\
     &&  + C \|\p_t{\mathbf g}\|_{L^2(\Omega)}^2
       + C \|\mathbf b\|_{L^2(\Omega)}^2 \|\nabla\mathbf b\|_{L^2(\Omega)}^2
           \|\p_t{\mathbf g}\|_{L^2(\Omega)}^2, \\
I_5
&\leq& C (\|w\p_t{\mathbf u}\|_{L^{\f43}(\Omega)}
        + \|{\mathbf u}\p_tw\|_{L^{\f43}(\Omega)})
         \|\p_t{\mathbf g}\|_{L^4(\Omega)}    \nonumber\\
&\leq& C (\|w\|_{L^4(\Omega)} \|\p_t\mathbf u\|_{L^2(\Omega)}
        + \|\mathbf u\|_{L^4(\Omega)} \|\p_tw\|_{L^2(\Omega)})
         \|\p_t{\mathbf g}\|_{L^2(\Omega)}^{\f12}
         \|\nabla\p_t{\mathbf g}\|_{L^2(\Omega)}^{\f12} \nonumber\\
&\leq& C \|w\|_{L^4(\Omega)} (\|\p_t{\mathbf g}\|_{L^2(\Omega)}
          + \|\p_tw\|_{L^2(\Omega)})\|\p_t{\mathbf g}\|_{L^2(\Omega)}^{\f12}
         \|\nabla\p_t{\mathbf g}\|_{L^2(\Omega)}^{\f12}  \nonumber\\
       && + C \|\mathbf u\|_{L^2(\Omega)}^{\f12}
         \|\nabla {\mathbf u}\|_{L^2(\Omega)}^{\f12} \|\p_tw\|_{L^2(\Omega)} \|\p_t{\mathbf g}\|_{L^2(\Omega)}^{\f12}
         \|\nabla\p_t{\mathbf g}\|_{L^2(\Omega)}^{\f12} \nonumber\\
&\leq& \f{\mu + \chi}{14} \|\nabla\p_t{\mathbf g}\|_{L^2(\Omega)}^2
       + C \|w\|_{L^4(\Omega)}^{\f43} \|\p_t {\mathbf g}\|_{L^2(\Omega)}^2
       + C \|w\|_{L^4(\Omega)}^2 \|\p_t w\|_{L^2(\Omega)}^2 \nonumber\\
     &&+ C \|\p_t{\mathbf g}\|_{L^2(\Omega)}^2
       + C \|\mathbf u\|_{L^2(\Omega)}^2 \|\nabla {\mathbf u}\|_{L^2(\Omega)}^2 \|\p_t{\mathbf g}\|_{L^2(\Omega)}^2 + C \|\p_t w\|_{L^2(\Omega)}^2, \\
\label{I67}
I_6 + I_7
&\leq& C \|\p_t\mathbf u\|_{L^2(\Omega)} \|\p_t\mathbf g\|_{L^2(\Omega)}
       + C \|\p_t\mathbf v\|_{L^2(\Omega)} \|\p_t\mathbf g\|_{L^2(\Omega)} \nonumber\\
&\leq& C (\|\p_t\mathbf g\|_{L^2(\Omega)} + \|\p_t w\|_{L^2(\Omega)})
        \|\p_t\mathbf g\|_{L^2(\Omega)}
        + C \|\p_t w\|_{L^2(\Omega)} \|\p_t\mathbf g\|_{L^2(\Omega)} \nonumber\\
&\leq& C \|\p_t\mathbf g\|_{L^2(\Omega)}^2 + C \|\p_t w\|_{L^2(\Omega)}^2.
\end{eqnarray*}
Plugging the estimates of $ I_i $, $ i = 1, 2,\cdots, 7 $, into \eqref{ut21}, we obtain
\begin{eqnarray}\label{gH2}
&& \f12\f{d}{dt}\|\p_t{\mathbf g}\|_{L^2(\Omega)}^2
     + \f{9(\mu + \chi)}{14}\| \nabla \p_t {\mathbf g} \|_{L^2(\Omega)}^2
        \nonumber\\
&\leq& C \|\p_t\mathbf g\|_{L^2(\Omega)}^2
         (1 + \|\nabla \mathbf u\|_{L^2(\Omega)}^4
             + \|\mathbf u\|_{L^2}^2 \|\nabla\mathbf u\|_{L^2}^2
     + \|\mathbf b\|_{L^2}^2 \|\nabla\mathbf b\|_{L^2}^2
     + \|w\|_{L^4(\Omega)}^{\f43})+ \f{\nu}5\|\nabla\p_t\mathbf b\|_{L^2(\Omega)}^2\nonumber\\
    &&+ C \|\p_t w\|_{L^2(\Omega)}^2 (1+\|\nabla \mathbf u\|_{L^2(\Omega)}^4  + \|\nabla \mathbf u\|_{L^2(\Omega)}^2 + \|w\|_{L^4(\Omega)}^2)
  + C \|\p_t\mathbf b\|_{L^2(\Omega)}^2\|\nabla \mathbf b\|_{L^2(\Omega)}^4.
\end{eqnarray}

$ \bf Step ~2 $.
Differentiating $ \eqref{eq1}_3 $ with respect to $ t $, and multiplying it with $ \p_t{\mathbf b} $, we can deduce
\begin{eqnarray}\label{bt21}
&&\f12\f{d}{dt} \|\p_t\mathbf b\|_{L^2(\Omega)}^2
     + \nu \|\nabla \p_t\mathbf b\|_{L^2(\Omega)}^2\nonumber\\
&=&-\int_{\Omega}\p_t\mathbf u \cdot \nabla\mathbf b \cdot \p_t\mathbf b dx+ \int_{\Omega} \p_t\mathbf b \cdot \nabla\mathbf u \cdot \p_t\mathbf b dx
      + \int_{\Omega} \mathbf b \cdot \nabla\p_t\mathbf u \cdot \p_t\mathbf b dx = \sum_{i=1}^3 J_i.
\end{eqnarray}

By making use of the same tools in the Step 1, we get
\begin{eqnarray}
J_1
&\leq& \|\p_t\mathbf u\|_{L^4(\Omega)} \|\nabla\mathbf b\|_{L^2(\Omega)}
       \|\p_t\mathbf b\|_{L^4(\Omega)}  \nonumber\\
&\leq& C \|\p_t\mathbf u\|_{L^2(\Omega)}^{\f12}
         \|\nabla\p_t\mathbf u\|_{L^2(\Omega)}^{\f12}
         \|\nabla\mathbf b\|_{L^2(\Omega)}
         \|\p_t\mathbf b\|_{L^2(\Omega)}^{\f12}
         \|\nabla\p_t\mathbf b\|_{L^2(\Omega)}^{\f12} \nonumber\\
&\leq& C \|\p_t\mathbf u\|_{L^2(\Omega)}
         \|\nabla\p_t\mathbf u\|_{L^2(\Omega)}
         \|\nabla\mathbf b\|_{L^2(\Omega)}^2
        + C \|\p_t\mathbf b\|_{L^2(\Omega)}
          \|\nabla\p_t\mathbf b\|_{L^2(\Omega)} \nonumber\\
&\leq& \f{\mu + \chi}{14} \|\nabla \p_t\mathbf g\|_{L^2(\Omega)}^2
    + \f{\nu}{10} \|\nabla \p_t\mathbf b\|_{L^2(\Omega)}^2
    + C \|\p_t\mathbf g\|_{L^2(\Omega)}^2 \|\nabla \mathbf b\|_{L^2(\Omega)}^4+ C \|\p_t w\|_{L^2(\Omega)}^2
         \nonumber\\
      && + C \|\p_t w\|_{L^2(\Omega)}^2 \|\nabla \mathbf b\|_{L^2(\Omega)}^4
         + C \|\p_t w\|_{L^2(\Omega)}^2
             \|\nabla \mathbf b\|_{L^2(\Omega)}^2+ C \|\p_t\mathbf b\|_{L^2(\Omega)}^2,\nonumber\\
J_2 + J_3
&\leq& \|\p_t\mathbf b\|^2_{L^4(\Omega)} \|\nabla\mathbf u\|_{L^2(\Omega)}
     + \|\mathbf b\|_{L^4(\Omega)}
       \|\nabla\p_t \mathbf u\|_{L^2(\Omega)}
       \|\p_t \mathbf b\|_{L^4(\Omega)}   \nonumber\\
&\leq& C \|\p_t\mathbf b\|_{L^2(\Omega)} \|\nabla\p_t\mathbf b\|_{L^2(\Omega)}
         \|\nabla \mathbf u\|_{L^2(\Omega)} \nonumber\\
      &&+ C \|\mathbf b\|_{L^2(\Omega)}^{\f12}
         \|\nabla\mathbf b\|_{L^2(\Omega)}^{\f12}
         \|\nabla\p_t\mathbf u\|_{L^2(\Omega)}
         \|\p_t\mathbf b\|_{L^2(\Omega)}^{\f12}
         \|\nabla\p_t\mathbf b\|_{L^2(\Omega)}^{\f12}
    \nonumber\\
&\leq& C \|\p_t\mathbf b\|_{L^2(\Omega)} \|\nabla\p_t\mathbf b\|_{L^2(\Omega)}
         \|\nabla \mathbf u\|_{L^2(\Omega)} \nonumber\\
      &&+C \|\mathbf b\|_{L^2(\Omega)}^{\f12}
         \|\nabla\mathbf b\|_{L^2(\Omega)}^{\f12}
         \|\nabla\p_t\mathbf g\|_{L^2(\Omega)}
         \|\p_t\mathbf b\|_{L^2(\Omega)}^{\f12}
         \|\nabla\p_t\mathbf b\|_{L^2(\Omega)}^{\f12} \nonumber\\
       &&+ C\|\mathbf b\|_{L^2(\Omega)}^{\f12}
         \|\nabla\mathbf b\|_{L^2(\Omega)}^{\f12}
         \|\p_t w\|_{L^2(\Omega)}
         \|\p_t\mathbf b\|_{L^2(\Omega)}^{\f12}
         \|\nabla\p_t\mathbf b\|_{L^2(\Omega)}^{\f12}   \nonumber\\
&\leq& \f{\mu + \chi}{14} \|\nabla\p_t\mathbf g\|_{L^2(\Omega)}^2
        + \f{\nu}{5} \|\nabla\p_t\mathbf b\|_{L^2(\Omega)}^2
        + C \|\p_t\mathbf b\|_{L^2(\Omega)}^2 \|\nabla \mathbf u\|_{L^2(\Omega)}^2
         \nonumber\\
      &&+ C \|\p_t \mathbf b\|_{L^2}^2 \|\mathbf b\|_{L^2}^2
            \|\nabla \mathbf b\|_{L^2}^2
        + C \|\p_t w\|_{L^2(\Omega)}^2.\nonumber
\end{eqnarray}
Plugging the estimates of $ J_i $, $ i = 1, 2, 3 $, into \eqref{bt21}, we have
\begin{eqnarray}\label{bH2}
&& \f12\f{d}{dt} \|\p_t\mathbf b\|_{L^2(\Omega)}^2
     + \f{7\nu}{10} \|\nabla \p_t\mathbf b\|_{L^2(\Omega)}^2  \nonumber\\
&\leq& \f{\mu + \chi}7 \|\nabla \p_t \mathbf g\|_{L^2(\Omega)}^2 +C \|\p_t\mathbf b\|_{L^2(\Omega)}^2 (1 + \|\nabla\mathbf u\|_{L^2(\Omega)}^2 + \|\mathbf b\|_{L^2(\Omega)}^2 \|\nabla\mathbf b\|_{L^2(\Omega)}^2) \nonumber\\
&& + C \|\p_t w\|_{L^2(\Omega)}^2 (1 + \|\nabla \mathbf b\|_{L^2(\Omega)}^4
   + \|\nabla \mathbf b\|_{L^2(\Omega)}^2 )
   + C \|\p_t\mathbf g\|_{L^2(\Omega)}^2\|\nabla \mathbf b\|_{L^2(\Omega)}^4.
\end{eqnarray}

Finally, adding up \eqref{gH2} and \eqref{bH2}, it follows that
\begin{eqnarray}\label{gbt2}
&&\f{d}{dt} (\|\p_t\mathbf g\|_{L^2(\Omega)}^2
             + \|\p_t\mathbf b\|_{L^2(\Omega)}^2)
  + (\mu + \chi)\|\nabla \p_t\mathbf g\|_{L^2(\Omega)}^2
  + \nu \|\nabla \p_t\mathbf b\|_{L^2(\Omega)}^2 \nonumber\\
&\leq& C (\|\p_t\mathbf g\|_{L^2(\Omega)}^2
             + \|\p_t\mathbf b\|_{L^2(\Omega)}^2)
         (1 + \|\nabla \mathbf u\|_{L^2(\Omega)}^4
             + \|\nabla \mathbf b\|_{L^2(\Omega)}^4 + \|w\|_{L^4(\Omega)}^{\f43} ) \nonumber\\
     && + C \|\p_t w\|_{L^2(\Omega)}^2 (1 + \|\nabla \mathbf u\|_{L^2(\Omega)}^4 + \|w\|_{L^4(\Omega)}^2  + \|\nabla \mathbf b\|_{L^2(\Omega)}^4 ).
\end{eqnarray}

$ \bf Step ~3 $.
Multiplying the inequality \eqref{gbt2} by $t$, we get
\begin{eqnarray*}\label{gbt21}
&&\f{d}{dt} (t\|\p_t\mathbf g\|_{L^2(\Omega)}^2
             + t\|\p_t\mathbf b\|_{L^2(\Omega)}^2)
  + (\mu + \chi)t \|\nabla \p_t\mathbf g\|_{L^2(\Omega)}^2
  + \nu t \|\nabla \p_t\mathbf b\|_{L^2(\Omega)}^2 \nonumber\\
&\leq& C t(\|\p_t\mathbf g\|_{L^2(\Omega)}^2
              + \|\p_t\mathbf b\|_{L^2(\Omega)}^2)
              (1 + \|\nabla \mathbf u\|_{L^2(\Omega)}^4
             + \|\nabla \mathbf b\|_{L^2(\Omega)}^4 + \|w\|_{L^4(\Omega)}^{\f43} )  \nonumber\\
       &&+ C t \|\p_t w\|_{L^2(\Omega)}^2 (1 + \|\nabla \mathbf u\|_{L^2(\Omega)}^4 + \|w\|_{L^4(\Omega)}^2  + \|\nabla \mathbf b\|_{L^2(\Omega)}^4)+ C (\|\p_t \mathbf g\|_{L^2(\Omega)}^2+ \|\p_t \mathbf b\|_{L^2(\Omega)}^2),\nonumber
\end{eqnarray*}
which implies, after applying Gr\"{o}nwall's inequality, Proposition \ref{uH1} and Lemma \ref{ptw}, that
\begin{eqnarray}\label{t2}
&&\sup\limits_{0 \leq t \leq T}t(\|\p_t\mathbf g\|_{L^2(\Omega)}^2
+ \|\p_t\mathbf b\|_{L^2(\Omega)}^2)
             + (\mu + \chi)\int_0^T t\|\nabla\p_t{\mathbf g}\|_{L^2(\Omega)}^2dt+\nu \int_0^T t\|\nabla\p_t{\mathbf b}\|_{L^2(\Omega)}^2 dt
\leq C.
\end{eqnarray}

Plugging \eqref{A105.1} into \eqref{stokes g}, we have
\begin{eqnarray}\label{nalba2gt}
&&(\mu + \chi)\|\nabla^2 \mathbf g\|_{L^2(\Omega)}^2\nonumber\\
&\leq&  C \|\p_t \mathbf g\|_{L^2(\Omega)}^2 + \f{\mu+\chi}{8} \|\nabla^2 \mathbf g\|_{L^2(\Omega)}^2+ \f{\mu + \chi}{4} \|\nabla^2 \mathbf b\|_{L^2(\Omega)}^2 + C (\|\mathbf u\|_{L^2(\Omega)}^2 + \|w\|_{L^2(\Omega)}^2)\nonumber\\
&&+ C (1 + \|\mathbf u\|_{L^2(\Omega)}^2 \|\nabla \mathbf u\|_{L^2(\Omega)}^2
  + \|\mathbf b\|_{L^2(\Omega)}^2 \|\nabla \mathbf b\|_{L^2(\Omega)}^2)
   (\|\nabla \mathbf g\|_{L^2(\Omega)}^2+ \|\nabla \mathbf b\|_{L^2(\Omega)}^2
   + \|w\|_{L^4(\Omega)}^2)
\end{eqnarray}
Similarly, by \eqref{stokes b1}--\eqref{A106.1}, there holds
\begin{eqnarray}\label{nalba2bt}
&&\nu\|\nabla^2 \mathbf b\|_{L^2(\Omega)}^2\nonumber\\
&\leq& C \|\p_t \mathbf b\|_{L^2(\Omega)}^2 + \f{\nu}8 \|\nabla^2 \mathbf g\|_{L^2(\Omega)}^2
+ \f{\nu}4 \|\nabla^2 \mathbf b\|_{L^2(\Omega)}^2 + C (1 + \|\mathbf u\|_{L^2(\Omega)}^2 \|\nabla \mathbf u\|_{L^2(\Omega)}^2+ \|\mathbf b\|_{L^2(\Omega)}^2 \|\nabla \mathbf b\|_{L^2(\Omega)}^2)\times\nonumber\\
  &&(\|\nabla \mathbf g\|_{L^2(\Omega)}^2 + \|\nabla \mathbf b\|_{L^2(\Omega)}^2
   + \|w\|_{L^4(\Omega)}^2).
\end{eqnarray}
Then by multiplying \eqref{nalba2gt} with $ \nu $, \eqref{nalba2bt} with $ \mu + \chi $, and adding the resultants up, one has
\begin{eqnarray}\label{nabla2gbt}
&&\f{3(\mu+\chi)\nu}4 \|\nabla^2 \mathbf g\|_{L^2(\Omega)}^2
     + \f{(\mu + \chi)\nu}2 \|\nabla^2 \mathbf b\|_{L^2(\Omega)}^2
 \nonumber\\
&\leq& C (\|\p_t \mathbf g\|_{L^2(\Omega)}^2 + \|\p_t \mathbf b\|_{L^2(\Omega)}^2)
+ C (1 + \|\mathbf u\|_{L^2(\Omega)}^2 \|\nabla \mathbf u\|_{L^2(\Omega)}^2
          + \|\mathbf b\|_{L^2(\Omega)}^2 \|\nabla \mathbf b\|_{L^2(\Omega)}^2) \times    \nonumber\\
  && (\|\nabla \mathbf g\|_{L^2(\Omega)}^2 + \|\nabla \mathbf b\|_{L^2(\Omega)}^2
   + \|w\|_{L^4(\Omega)}^2)
+ C (\|\mathbf u\|_{L^2(\Omega)}^2 + \|w\|_{L^2(\Omega)}^2),
\end{eqnarray}
which yields, after multiplying \eqref{nabla2gbt} by $ t $, that
\begin{eqnarray}
&& \f{3(\mu+\chi)\nu}4 \, t \, \|\nabla^2 \mathbf g\|_{L^2(\Omega)}^2
     + \f{(\mu + \chi)\nu}2 \, t \, \|\nabla^2 \mathbf b\|_{L^2(\Omega)}^2  \nonumber\\
&\leq& C t(\|\p_t \mathbf g\|_{L^2(\Omega)}^2 + \|\p_t \mathbf b\|_{L^2(\Omega)}^2)
+ C t(1 + \|\mathbf u\|_{L^2(\Omega)}^2 \|\nabla \mathbf u\|_{L^2(\Omega)}^2
          + \|\mathbf b\|_{L^2(\Omega)}^2 \|\nabla \mathbf b\|_{L^2(\Omega)}^2) \times    \nonumber\\
  && (\|\nabla \mathbf g\|_{L^2(\Omega)}^2 + \|\nabla \mathbf b\|_{L^2(\Omega)}^2
   + \|w\|_{L^4(\Omega)}^2)
+ C t(\|\mathbf u\|_{L^2(\Omega)}^2 + \|w\|_{L^2(\Omega)}^2).
\end{eqnarray}

Thus, by combing \eqref{t2}, Proposition \ref{uH1} and Lemma \ref{gH1}, we can obtain
\begin{eqnarray}\label{tnabla2gb}
\sup\limits_{0 \leq t \leq T} \, t(\|\nabla^2{\mathbf g}\|_{L^2(\Omega)}^2
  + \|\nabla^2{\mathbf b}\|_{L^2(\Omega)}^2)
\leq C,
\end{eqnarray}
which further implies, after adding \eqref{t2}, that
\begin{eqnarray}\label{tt}
&&\sup\limits_{0 \leq t \leq T} \, t(\|\nabla^2{\mathbf g}\|_{L^2(\Omega)}^2
  + \|\nabla^2{\mathbf b}\|_{L^2(\Omega)}^2 + \|\p_t{\mathbf g}\|_{L^2(\Omega)}^2 + \|\p_t{\mathbf b}\|_{L^2(\Omega)}^2)  \nonumber\\
&&+ \int_0^T t (\|\nabla\p_t{\mathbf g}\|_{L^2(\Omega)}^2
  + \|\nabla\p_t{\mathbf b}\|_{L^2(\Omega)}^2) \, dt
\leq C.
\end{eqnarray}
All that's left is to deduce the corresponding estimates of $\mathbf u$. By noticing $ \mathbf g = \mathbf u - \mathbf v $, applying the inequality \eqref{keyestimates}, H\"{o}lder's inequality and Lemmas \ref{nablaw-Lp}--\ref{ptw}, we have
\begin{eqnarray*}
&&t\, \|\nabla^2 \mathbf u\|_{L^2(\Omega)}^2
\leq t\, (\|\nabla^2 \mathbf g\|_{L^2(\Omega)}^2
     + \|\nabla^2 \mathbf v\|_{L^2(\Omega)}^2) \\
&\leq& C t\, (\|\nabla^2 \mathbf g\|_{L^2(\Omega)}^2 + \|\nabla w\|_{L^2(\Omega)}^2)\leq C t\, (\|\nabla^2 \mathbf g\|_{L^2(\Omega)}^2 + 1),
\end{eqnarray*}
and
\begin{eqnarray*}
&&t\, \|\nabla\p_t\mathbf u\|_{L^2(\Omega)}^2\leq t\, (\|\nabla\p_t\mathbf g\|_{L^2(\Omega)}^2
   + \|\nabla\p_t\mathbf v\|_{L^2(\Omega)}^2)  \\
&\leq& Ct\, (\|\nabla\p_t\mathbf g\|_{L^2(\Omega)}^2 + \| \p_t w\|_{L^2(\Omega)}^2)\leq Ct\, (\|\nabla\p_t\mathbf g\|_{L^2(\Omega)}^2 + 1).
\end{eqnarray*}
Recalling \eqref{tt}, we finally obtain
\begin{eqnarray*}
\sup\limits_{0 \leq t \leq T}\, t \, \|\nabla^2 \mathbf u\|_{L^2(\Omega)}^2+\int_0^T t \, \|\nabla\p_t\mathbf u\|_{L^2(\Omega)}^2
\leq C,
\end{eqnarray*}
which completes the proof of Proposition \ref{ubtH2}.
\end{proof}

With the global bounds obtained in Proposition \ref{ubtH2}, we can conclude that weak solutions obtained in Section \ref{weak solutions} are actually strong solutions, which clearly implies Theorem \ref{T3}.
~\\
\subsection{Uniqueness of strong solutions}
~\\
\label{section4.2}

For the proof of Theorem \ref{T1}, it remains to certify the uniqueness of strong solutions.

\begin{proof}[\bf Uniqueness:]
Suppose that $ ({\mathbf u}, w, {\mathbf b}) $ and $ (\widetilde{\mathbf u}, \widetilde{w}, \widetilde{\mathbf b}) $  are two strong solutions of system \eqref{eq1}--\eqref{eq20} with the regularity specified in Proposition \ref{uH1} and Proposition \ref{ubtH2}.
Setting ${\mathbf U} = {\mathbf u} - \widetilde{\mathbf u},\, W = w - \widetilde{w}, \,{\mathbf B} = {\mathbf b} - \widetilde{\mathbf b}, \,\Pi = p - \widetilde{p}$, then $ (\mathbf U, W, \mathbf B, \Pi) $ satisfies
\begin{equation}\label{uni}
\left\{\begin{array}{ll}
{\mathbf U}_t + {\mathbf u}\cdot\nabla {\mathbf U} + {\mathbf U}\cdot\nabla \widetilde{\mathbf u} + \nabla \Pi
= (\mu+\chi )\Delta{\mathbf U} + \mathbf b \cdot \nabla\mathbf B + \mathbf B \cdot \nabla \widetilde{\mathbf b}- \chi{\nabla^{\perp}}W,\vspace{1ex}\\
W_t + {\mathbf u} \cdot \nabla W + {\mathbf U} \cdot \nabla \widetilde{w} + 2 \chi W = \chi {\nabla^{\perp}}\cdot {\mathbf U},\vspace{1ex}\\
{\mathbf B}_t + {\mathbf u}\cdot\nabla {\mathbf B} + {\mathbf U}\cdot\nabla \widetilde{\mathbf b}
= \nu \Delta{\mathbf B} + \mathbf b \cdot \nabla\mathbf U + \mathbf B \cdot \nabla \widetilde{\mathbf u},\vspace{1ex}\\
\nabla\cdot {\mathbf U}=0=\nabla\cdot {\mathbf B}=0,\quad {\mathbf U}|_{\p\Omega}={\mathbf B}|_{\p\Omega}=0,  \vspace{1ex}\\
({\mathbf U},W, \mathbf B)(x,0)=0.
\end{array}\right.
\end{equation}

Multiplying the first three equations of \eqref{uni} with $ \mathbf U, W, \mathbf B $ respectively yields
\begin{eqnarray}\label{UWB}
&&\f12\f{d}{dt}(\|\mathbf U\|_{L^2(\Omega)}^2 + \|W\|_{L^2(\Omega)}^2
    + \|\mathbf B\|_{L^2(\Omega)}^2)
     + (\mu + \chi)\|\nabla \mathbf U\|_{L^2(\Omega)}^2
     + \nu\|\nabla \mathbf B\|_{L^2(\Omega)}^2
      + 2\chi \|W\|_{L^2(\Omega)}^2   \nonumber\\
&=&  - \chi \int_\Omega \nabla^\bot W \cdot \mathbf U \,dx
     + \chi \int_\Omega \nabla^\bot \cdot \mathbf U W \,dx
     - \int_\Omega \mathbf U \cdot \nabla \widetilde{u} \cdot \mathbf U \,dx
     + \int_\Omega \mathbf B \cdot \nabla \widetilde{b} \cdot \mathbf U \,dx
 \nonumber\\
    && - \int_\Omega \mathbf U \cdot \nabla \widetilde{w} \cdot W \,dx
     - \int_\Omega \mathbf U \cdot \nabla\widetilde{\mathbf b} \cdot\mathbf B\,dx
     + \int_\Omega \mathbf B \cdot \nabla\widetilde{\mathbf u} \cdot\mathbf B\,dx
     = \sum_{i=1}^7 I_i.
\end{eqnarray}

For the seven terms, by using the boundary conditions $ \mathbf U|_{\p\Omega} = 0 $, H\"{o}lder's inequality, Corollary \ref{C1}, Young's inequality, one gets
\begin{eqnarray*}
I_1 + I_2
&=& 2\chi \int_\Omega \nabla^\bot \cdot \mathbf U W \,dx \leq \f{\mu + \chi}4 \|\nabla \mathbf U \|_{L^2(\Omega)}^2
     + C \|W\|_{L^2(\Omega)}^2, \\
\sum_{i=3}^7I_i
&\leq& \|\nabla \widetilde{\mathbf u}\|_{L^2(\Omega)}
       \|\mathbf U\|_{L^4}^2
  + \|\nabla \widetilde{\mathbf b}\|_{L^2(\Omega)} \|\mathbf B\|_{L^4}
    \|\mathbf U\|_{L^4(\Omega)}
  + \|\nabla \widetilde{w}\|_{L^4} \|\mathbf U\|_{L^4(\Omega)}\|W\|_{L^2}\\
 && + \|\nabla\widetilde{\mathbf b}\|_{L^2(\Omega)} \|\mathbf U\|_{L^4(\Omega)} \|\mathbf B\|_{L^4(\Omega)} + \|\nabla \widetilde{\mathbf u}\|_{L^2(\Omega)}
  \|\mathbf B\|_{L^4}^2 \\
&\leq& C \|\nabla \widetilde{\mathbf u}\|_{L^2(\Omega)}
         \|\mathbf U\|_{L^2(\Omega)} \|\nabla \mathbf U\|_{L^2(\Omega)}
    + C \|\nabla \widetilde{w}\|_{L^4} \|\mathbf U\|_{L^2(\Omega)}^{\f12}
       \|\nabla \mathbf U\|_{L^2(\Omega)}^{\f12} \|W\|_{L^2}  \\
   &&  + C \|\nabla\widetilde{\mathbf b}\|_{L^2(\Omega)}
       \|\mathbf B\|_{L^2(\Omega)}^{\f12}\|\nabla \mathbf B\|_{L^2(\Omega)}^{\f12} \|\mathbf U\|_{L^2(\Omega)}^{\f12}\|\nabla \mathbf U\|_{L^2(\Omega)}^{\f12}
    + C \|\nabla \widetilde{\mathbf u}\|_{L^2(\Omega)}
         \|\mathbf B\|_{L^2(\Omega)} \|\nabla \mathbf B\|_{L^2(\Omega)}\\
&\leq& \f{\mu + \chi}{4}\|\nabla \mathbf U\|_{L^2(\Omega)}^2
        + \f{\nu}2\|\nabla \mathbf B\|_{L^2(\Omega)}^2
        + C \|\nabla\widetilde{\mathbf u}\|_{L^2(\Omega)}^2 \|\mathbf U\|_{L^2(\Omega)}^2 + C \|\nabla\widetilde{w}\|_{L^4(\Omega)}^2\|W\|_{L^2(\Omega)}^2 \\
   && + C \|\nabla \widetilde{\mathbf b}\|_{L^2(\Omega)}^2
           (\|\mathbf B\|_{L^2(\Omega)}^2+ \|\mathbf U\|_{L^2(\Omega)}^2)
      + C\|\mathbf U\|_{L^2(\Omega)}^2 + C \|\nabla\widetilde{\mathbf u}\|_{L^2(\Omega)}^2 \|\mathbf B\|_{L^2(\Omega)}^2.
\end{eqnarray*}
Plugging the above estimates into \eqref{UWB}, we obtain
\begin{eqnarray*}
&&\f{d}{dt}(\|\mathbf U\|_{L^2(\Omega)}^2 + \|W\|_{L^2(\Omega)}^2
    + \|\mathbf B\|_{L^2(\Omega)}^2)
     + (\mu + \chi)\|\nabla \mathbf U\|_{L^2(\Omega)}^2
     + \nu\|\nabla \mathbf B\|_{L^2(\Omega)}^2 + 2\chi \|W\|_{L^2(\Omega)}^2\\
&\leq& C (1+ \|\nabla \widetilde{\mathbf u}\|_{L^2(\Omega)}^2
    + \|\nabla \widetilde{\mathbf b}\|_{L^2(\Omega)}^2
    + \|\nabla \widetilde{w}\|_{L^4(\Omega)}^2)
    (\|\mathbf U\|_{L^2(\Omega)}^2 + \|W\|_{L^2(\Omega)}^2 + \|\mathbf B\|_{L^2(\Omega)}^2),
\end{eqnarray*}
which implies, after applying Gr\"{o}nwall's inequality and Proposition \ref{uH1}, that
\begin{eqnarray*}
  &&\|\mathbf U\|_{L^2(\Omega)}^2 + \|W\|_{L^2(\Omega)}^2
    + \|\mathbf B\|_{L^2(\Omega)}^2 \\
&\leq& e^{C\int_0^t (1 + \|\nabla \widetilde{\mathbf u}\|_{L^{2}(\Omega)}^2 + \|\nabla \widetilde{\mathbf b}\|_{L^{2}(\Omega)}^2 + \|\nabla\widetilde{w}\|_{L^{4}(\Omega)}^2) \,d\tau}(\|{\mathbf U}_0\|_{L^2(\Omega)}^2 + \|W_0\|_{L^2(\Omega)}^2
    + \|{\mathbf B}_0\|_{L^2(\Omega)}^2)  \\
&\leq& e^{CT} (\|{\mathbf U}_0\|_{L^2(\Omega)}^2 + \|W_0\|_{L^2(\Omega)}^2
    + \|{\mathbf B}_0\|_{L^2(\Omega)}^2).
\end{eqnarray*}
Thus we obtain $ {\mathbf U} = W = \mathbf B \equiv 0 $ according to $ {\mathbf U}_0 = W_0 = {\mathbf B}_0 = 0 $. This finishes the proof of Theorem \ref{T1}.

\end{proof}

\section*{Acknowledgments}
S. S. Wang is supported by Beijing University of Technology (No. ykj-2018-00110).
J. T. Liu is supported by National Natural Science Foundation of China (No.  11801018), Beijing Natural Science Foundation (No. 1192001), Youth Backbone Individual Program of the Organization Department of Beijing Municipality (No.  2017000020124G052) and Beijing University of Technology (No. 006000514120513).

\end{document}